\documentclass[12pt,a4paper]{article}
\usepackage[utf8]{inputenc}
\usepackage{amsmath,amsthm}
\usepackage{amsfonts}
\usepackage{amssymb}
\usepackage{makeidx}
\usepackage{graphicx}
\newcommand{\n}{\noindent}

\newtheorem{theorem}{Theorem}[section]
\newtheorem{proposition}{Proposition}[section]

\newtheorem{lemma}{Lemma}[section]

\newcommand{\Mod}[1]{\ (\mathrm{mod}\ #1)}
\author{Ratan Lal and Vipul Kakkar}
\title{Automorphisms of Zappa-Sz\'{e}p Product}

\begin{document}
	\maketitle
	\begin{abstract}
	\n	In this paper, we have found the automorphism group of the Zappa-Sz\'{e}p product of two groups. Also, we have computed the automorphism group of the Zappa-Sz\'{e}p product of a cyclic group of order $m$ and a cyclic group of order $p^{2}$, where $p$ is a prime.
	\end{abstract}
\n \textbf{Keywords:} {Zappa-Sz\'{e}p Product, Automorphism Group.}
	\section{Introduction}
	\n	A group $G$ is called the internal Zappa-Sz\'{e}p product of its two subgroups $H$ and $K$ if $G=HK$ and $H\cap K = \{1\}$. The Zappa-Sz\'{e}p product is a natural generalization of the semidirect product of two groups in which neither of the factor is required to be normal. If $G$ is the internal Zappa-Sz\'{e}p product of $H$ and $K$, then $K$ appears as a right transversal to $H$ in $G$. Let $h\in H$ and $k\in K$. Then $kh=\sigma(x,h)\theta(x,h)$, where $\sigma(x,h)\in H$ and $\theta(x,h)\in K$. This determines the maps $\sigma: K \times H \rightarrow H$ and $\theta: K\times H \rightarrow K$. These maps are called the matched pair of groups. We denote $\sigma(k,h)= k\cdot h$ and $\theta(k,h) = k^{h}$. These maps satisfy the following conditions (See \cite{af})
	
	\begin{itemize}
		\item[$(C1)$] $1\cdot h = h$ and $k^{1} = k$,
		\item[($C2$)] $k\cdot 1 = 1 = 1^{h}$,
		\item[$(C3)$] $kk^{\prime}\cdot h = k\cdot (k^{\prime}\cdot h)$,
		\item[($C4$)] $(kk^{\prime})^{h} = k^{k^{\prime}\cdot h}{k^{\prime}}^{h}$,
		\item[($C5$)] $k\cdot (hh^{\prime}) = (k\cdot h)(k^{h}\cdot h^{\prime})$,
		\item[$(C6)$] $k^{hh^{\prime}} = (k^{h})^{h^{\prime}}$,
	\end{itemize} 
	for all $h,h^{\prime} \in H$ and $k,k^{\prime}\in K$.

	\vspace{.2cm}
	
	\n On the other hand, let $H$ and $K$ are two groups. Let $\sigma: K \times H \rightarrow H$ and $\theta: K\times H \rightarrow K$ be two maps defined by $\sigma(k,h)= k\cdot h$ and $\theta(k,h) = k^{h}$ which satisfy the above conditions. Then, the external Zappa-Sz\'{e}p product $G=H\bowtie K$ of $H$ and $K$ is the group defined on the set $H\times K$ with the binary operation defined by
	\begin{equation*}
	(h,k)(h^{\prime},k^{\prime}) = (h(k\cdot h^{\prime}),k^{h^{\prime}}k^{\prime}).
	\end{equation*}
	
	\n The internal Zappa-Sz\'{e}p product is isomorphic to the external Zappa-Sz\'{e}p product (see \cite[Proposition 2.4, p. 4]{af}). We will identify the external Zappa-Sz\'{e}p product with the internal Zappa-Sz\'{e}p product. 
	
	\vspace{.2cm}
	
	\n The Zappa-Sz\'{e}p product of two groups was introduced by G. Zappa in \cite{gz}. J. Sz\'{e}p in the series of papers (few of them are \cite{sz1,sz2,sz3,sz4}) studied such type of products. From the QR decomposition of matrices, one concludes that the general linear group $GL(n,\mathbb{C})$ is a Zappa-Sz\'{e}p product of the unitary group and the group of upper triangular matrices. Z. Arad and E. Fisman in \cite{za} studied the finite simple groups as a Zappa-Sz\'{e}p product of two groups $H$ and $K$ with the order of $H$ and $K$ coprime. In the same paper, they studied the finite simple groups as a Zappa-Sz\'{e}p product of two groups $H$ and $K$ with one of $H$ or $K$ is $p$-group, where $p$ is a prime. From the main result of \cite{ph}, one observes that a finite group $G$ is solvable if and only if $G$ is a Zappa-Sz\'{e}p product of a Sylow $p$-subgroup and a Sylow $p$-complement.
	
	\vspace{.2cm}
	
	\n Note that, if either of the actions $k\cdot h$ or $k^h$ is a group homomorphism, then the Zappa-Sz\'{e}p product reduces to the semidirect product of groups.  M. J. Curran \cite{crn2008} and N. C. Hsu \cite{nch} studied the automorphisms of the semidirect product of two groups as the $2\times 2$ matrices of maps satisfying some certain conditions. In this paper (with the same terminology as in \cite{crn2008} and \cite{nch}), we have found the automorphism group of the Zappa-Sz\'{e}p product of two groups as the $2\times 2$ matrices of maps satisfying some certain conditions. As an application, we have found the automorphism group of the Zappa-Sz\'{e}p product of two cyclic groups in which one is of order $p^2$ and other is of order $m$. Throughout the paper, $\mathbb{Z}_{n}$ denotes the cyclic group of order $n$ and $U(n)$ denotes the group of units of $n$. Also, $Aut(G)$ denotes the group of all automorphisms of a group $G$. Let $U$ and $V$ be groups. Then $CrossHom(U,V)$ denotes the group of all crossed homomorphisms from $U$ to $V$. Also, if $U$ acts on $V$, then $Stab_{U}(V)$ denotes the stabilizer of $V$ in $U$.
	
	
	\section{Structure of Automorphism Group}
	\n  Let $G = H\bowtie K$ be the Zappa-Sz\'{e}p product of two groups $H$ and $K$. Let $U, V$ and $W$ be any groups. $Map(U, V)$ denotes the set of all maps between the groups $U$ and $V$. If $\phi, \psi \in Map(U, V)$ and $\eta \in Map(V, W)$, then $\phi + \psi \in Map (U,V)$ is defined by $(\phi + \psi)(u) = \phi(u)\psi(u)$, $\eta\phi \in Map(U,W)$ is defined by $\eta\phi(u) = \eta(\phi(u))$, $\phi \cdot \psi \in Map(U,V)$ is defined by $(\phi \cdot \psi)(u) = \phi(u)\cdot \psi(u)$ and $\phi^{\psi} \in Map(U, V)$ is defined by $\phi^{\psi}(u) = \phi(u)^{\psi(u)}$, for all $u\in U$.
	
	\vspace{.2cm}
	
	\n Now, consider the set, \begin{equation*}
	\mathcal{A} = \left\{\begin{pmatrix}
	\alpha & \beta\\
	\gamma & \delta
	\end{pmatrix} \;|\; \begin{matrix}
	\alpha \in Map(H,H), & \beta \in Map(K,H),\\
	\gamma \in Map(H,K),\; \text{and} &  \delta \in Map(K,K)
	\end{matrix}  \right\},
	\end{equation*}
	where $\alpha, \beta, \gamma$ and $\delta$ satisfy the following conditions,
	\begin{itemize}
		\item[$(A1)$] $\alpha(hh^{\prime}) = \alpha(h)(\gamma(h)\cdot \alpha(h^{\prime}))$,
		\item[$(A2)$] $\gamma(hh^{\prime}) = \gamma(h)^{\alpha(h^{\prime})}\gamma(h^{\prime})$,
		\item[$(A3)$] $\beta(kk^{\prime}) = \beta(k)(\delta(k)\cdot \beta(k^{\prime}))$,
		\item[$(A4)$] $\delta(kk^{\prime}) = \delta(k)^{\beta(k^{\prime})}\delta(k^{\prime})$,
		\item[$(A5)$] $\beta(k)(\delta(k)\cdot \alpha(h)) = \alpha(k\cdot h)(\gamma(k\cdot h)\cdot \beta(k^{h}))$,
		\item[$(A6)$] $\delta(k)^{\alpha(h)}\gamma(h) = \gamma(k\cdot h)^{\beta(k^{h})}\delta(k^{h})$, 
		\item[$(A7)$] For any $h^{\prime}k^{\prime}\in G$, there exists a unique $h\in H$ and $k\in K$ such that $h^{\prime} = \alpha(h)(\gamma(h)\cdot \beta(k))$ and $k^{\prime} = \gamma(h)^{\beta(k)}\delta(k)$.
	\end{itemize}
	\n Then, the set $\mathcal{A}$ forms a group with the binary operation defined by
	\begin{equation*}
	\begin{pmatrix}
	\alpha^{\prime} & \beta^{\prime}\\
	\gamma^{\prime} & \delta^{\prime}
	\end{pmatrix}\begin{pmatrix}
	\alpha & \beta\\
	\gamma & \delta
	\end{pmatrix} = \begin{pmatrix}
	\alpha^{\prime}\alpha + \gamma^{\prime}\alpha \cdot \beta^{\prime}\gamma & \alpha^{\prime}\beta + \gamma^{\prime}\beta \cdot \beta^{\prime}\delta\\
	(\gamma^{\prime}\alpha)^{\beta^{\prime}\gamma}+ \delta^{\prime}\gamma & (\gamma^{\prime}\beta)^{\beta^{\prime}\delta}+ \delta^{\prime}\delta
	\end{pmatrix}.
	\end{equation*} 
	\begin{proposition}
		Let $\begin{pmatrix}
		\alpha & \beta\\
		\gamma & \delta
		\end{pmatrix}\in \mathcal{A}$. Then $\alpha(1) = 1= \beta(1) = \gamma(1) = \delta(1)$. 
	\end{proposition}
	\begin{proof}
		Let $h\in H$ be any element. Then, using $(A1)$, $\alpha(h) = \alpha(h1) = \alpha(h)(\gamma(h)\cdot \alpha(1)) $ which implies that $\gamma(h)\cdot \alpha(1) = 1 = \gamma(h)\cdot 1$ by $(C2)$. Thus $\gamma(h)^{-1}\cdot(\gamma(h)\cdot \alpha(1)) = \gamma(h)^{-1}\cdot(\gamma(h)\cdot 1)$. Hence, using $(C1)$, $\alpha(1) = 1$. 
		
		\vspace{0.2 cm}
		
		\n Using $(A2)$, $\gamma(h) = \gamma(h1) = \gamma(h)^{\alpha(1)}\gamma(1)$. Using $(C1)$, $\gamma(1)=1$. Using the similar argument, we get $\beta(1)=1$ and $\delta(1)=1$.
	\end{proof}
	
	\n 	Let us define the kernel of the map $\alpha\in Map(H,H)$ as usual, that is, $\ker(\alpha) = \{h\in H \;|\; \alpha(h) = 1\}$. Here, we should remember that the map $\alpha$ need not to be a homomorphism. $\ker(\beta), \ker(\gamma)$ and $\ker(\delta)$ are defined in the same sense.
	\begin{lemma}\label{lk1}
		\begin{align*}
	\begin{matrix}
		(i) \ker(\alpha)\le H, & (ii) \ker(\beta)\le K,\\
		(iii) \ker(\gamma)\le H, & (iv) \ker(\delta)\le K,\\
\hspace{1.8cm}(v) \ker(\alpha)\cap \ker(\gamma) = \{1\},& \hspace{1.8cm} (vi) \ker(\beta)\cap \ker(\delta) = \{1\}.
	\end{matrix}
	\end{align*}
	\end{lemma}
	\begin{proof}
		\begin{itemize}
		\item[$(i)$] Let $h$, $h^{\prime}\in \ker(\alpha)$. Then using $(A1)$ and $(C2)$, $\alpha(hh^{\prime}) = \alpha(h)(\gamma(h)\cdot \alpha(h^{\prime})) = \gamma(h)\cdot 1 = 1$. Also, $1= \alpha(1) = \alpha(h^{-1}h) = \alpha(h^{-1})(\gamma(h^{-1})\cdot 1)$. Thus, $\alpha(h^{-1}) = 1$. Hence, $hh^{\prime}$, $h^{-1}\in \ker(\alpha)$ and so $\ker(\alpha)\le H$.
			\item[$(ii)$] One can easily prove it using the similar argument as in the part $(i)$.
			\item[$(iii)$] Let $h$, $h^{\prime} \in \ker(\gamma)$. Then using $(A2)$ and $(C2)$, $\gamma(hh^{\prime}) = \gamma(h)^{\alpha(h^{\prime})} \gamma(h^{\prime})$ $= 1^{\alpha(h^{\prime})} = 1$. Also, $1 = \gamma(1) = \gamma(hh^{-1}) = \gamma(h)^{\alpha(h^{-1})} \gamma(h^{-1})$. Then, $\gamma(h^{-1}) = 1$ and so, $hh^{\prime}$, $h^{-1} \in \ker(H)$. Hence, $\ker(\gamma) \le H$.
			\item[$(iv)$] One can easily prove it using the similar argument as in the part $(iii)$.
			\item[$(v)$] Let $h\in \ker(\alpha)\cap \ker(\gamma)$. Then $\alpha(h)= 1 = \gamma(h)$. Therefore, $\theta(h) = 1$. Since $\theta \in Aut(G)$, $h=1$. Hence, $(v)$ holds.
			\item[$(vi)$] One can easily prove it using the similar argument as in the part $(v)$.
		\end{itemize}
	\end{proof}
	
	\begin{theorem}
		Let $G = H\bowtie K$ be the Zappa-Sz\'{e}p product of two groups $H$ and $K$, and $\mathcal{A}$ be as above. Then there is an isomorphism of groups between $Aut(G)$ and $\mathcal{A}$ given by $\theta \longleftrightarrow \begin{pmatrix}
		\alpha & \beta\\
		\gamma & \delta
		\end{pmatrix}$, where $\theta(h) =  \alpha(h)\gamma(h)$ and $\theta(k) = \beta(k)\delta(k)$, for all $h\in H$ and $k\in K$.
	\end{theorem}
	\begin{proof}
		
		\n	Let $\theta \in Aut(G)$ be defined by $\theta(h) = \alpha(h)\gamma(h)$ and $\theta(k) = \beta(k)\delta(k)$, for all $h\in H$ and $k\in K$. Now, for all $h,h^{\prime}\in H$, $\theta(hh^{\prime}) = \theta(h)\theta(h^{\prime}) = \alpha(h)\gamma(h)\alpha(h^{\prime})\gamma(h^{\prime}) = \alpha(h)(\gamma(h)\cdot \alpha(h^{\prime}))\gamma(h)^{\alpha(h^{\prime})}\gamma(h^{\prime})$. Thus, $\alpha(hh^{\prime})\gamma(hh^{\prime}) = (\alpha(h)(\gamma(h)\cdot \alpha(h^{\prime})))(\gamma(h)^{\alpha(h^{\prime})}\gamma(h^{\prime}))$. Therefore, by uniqueness of representation, we have $(A1)$ and $(A2)$. By the similar argument, we get $(A3)$ and $(A4)$. 
		
		\vspace{.2cm}
		
\n 	Now, $\theta(kh) = \theta((k\cdot h)(k^{h})) = \theta(k\cdot h)\theta(k^{h}) = \alpha(k\cdot h)\gamma(k\cdot h)\beta(k^{h})\delta(k^{h}) = \alpha(k\cdot h)(\gamma(k\cdot h)\cdot \beta(k^{h}))\gamma(k\cdot h)^{\beta(k^{h})}\delta(k^{h})$. Also, $\theta(kh) = \theta(k)\theta(h) = \beta(k)\delta(k)$ $\alpha(h)\gamma(h)$ $= \beta(k)(\delta(k)\cdot \alpha(h))\delta(k)^{\alpha(h)}\gamma(h)$. Therefore, by the uniqueness of representation, $\beta(k)(\delta(k)\cdot \alpha(h)) = \alpha(k\cdot h)(\gamma(k\cdot h)\cdot \beta(k^{h}))$ and $\delta(k)^{\alpha(h)}\gamma(h) = \gamma(k\cdot h)^{\beta(k^{h})}\delta(k^{h})$, which proves $(A5)$ and $(A6)$. Finally, $(A7)$ holds because $\theta$ is onto. Thus, to every $\theta \in Aut(G)$ we can associate the matrix $\begin{pmatrix}
		\alpha & \beta\\
		\gamma & \delta
		\end{pmatrix} \in \mathcal{A}$. This defines a map $T: Aut(G)\longrightarrow \mathcal{A}$ given by $\theta \longmapsto \begin{pmatrix}
		\alpha & \beta\\
		\gamma & \delta
		\end{pmatrix}$. 
		
		\vspace{.2cm}
		
		\n	Now, if $\begin{pmatrix}
		\alpha & \beta\\
		\gamma & \delta
		\end{pmatrix} \in \mathcal{A}$ satisfying the conditions $(A1)-(A7)$, then we associate to it, the map $\theta : G\longrightarrow G$ defined by $\theta(h) = \alpha(h)\gamma(h)$ and $\theta(k) = \beta(k)\delta(k)$, for all $h\in H$ and $k\in K$.
		Using $(A1)-(A6)$, one can check that $\theta$ is an endomorphism of $G$. Also, by $(A7)$, the map $\theta$ is onto. Now, let $hk\in \ker(\theta)$. Then $\theta(hk) = 1$. Therefore, $\alpha(h)(\gamma(h)\cdot \beta(k))\gamma(h)^{\beta(k)}\delta(k) = 1$ and so, by the uniqueness of representation $\alpha(h)(\gamma(h)\cdot \beta(k)) = 1$ and $\gamma(h)^{\beta(k)}\delta(k) = 1$. Again, by the uniqueness of representation and using $(C1), (C2), (C3)$ and $(C6)$, we get $\alpha(h) = 1 = \gamma(h)$ and $\beta(k) = 1 = \delta(k)$. Therefore, by the Lemma \ref{lk1} $(v)$ and $(vi)$, $h = 1 = k$ and so, $\ker(\theta) = \{1\}$. Thus, $\theta$ is one-one and hence, $\theta \in Aut(G)$. Thus, $T$ is a bijection. Let $\alpha,\beta, \gamma$ and $\delta$ be the maps associated with $\theta$ and $\alpha^{\prime},\beta^{\prime}, \gamma^{\prime}$ and $\delta^{\prime}$ be the maps associated with $\theta^{\prime}$. Now, for all $h\in H$ and $k\in K$, we have
		\begin{eqnarray*}
			\theta^{\prime}\theta(h) &=& \theta^{\prime}(\alpha(h)\gamma(h))\\
			&=& \alpha^{\prime}(\alpha(h))\gamma^{\prime}(\alpha(h))\beta^{\prime}(\gamma(h))\delta^{\prime}(\gamma(h))\\
			&=& \alpha^{\prime}(\alpha(h))(\gamma^{\prime}(\alpha(h))\cdot\beta^{\prime}(\gamma(h)))\gamma^{\prime}(\alpha(h))^{\beta^{\prime}(\gamma(h))}\delta^{\prime}(\gamma(h))\\
			&=& (\alpha^{\prime}\alpha+(\gamma^{\prime}\alpha\cdot\beta^{\prime}\gamma))(h) ((\gamma^{\prime}\alpha)^{\beta^{\prime}\gamma}+\delta^{\prime}\gamma)(h).
		\end{eqnarray*}
		\n Therefore, if we write $hk$ as $\begin{pmatrix}
		h\\
		k
		\end{pmatrix}$, then $\theta^{\prime}\theta(h) = \begin{pmatrix}
		\alpha^{\prime}\alpha+(\gamma^{\prime}\alpha\cdot\beta^{\prime}\gamma)\\
		(\gamma^{\prime}\alpha)^{\beta^{\prime}\gamma}+\delta^{\prime}\gamma
		\end{pmatrix}\begin{pmatrix}
		h\\
		1
		\end{pmatrix}$. By the similar argument, $\theta^{\prime}\theta(k) = \begin{pmatrix}
		\alpha^{\prime}\beta + \gamma^{\prime}\beta \cdot \beta^{\prime}\delta\\
		(\gamma^{\prime}\beta)^{\beta^{\prime}\delta}+ \delta^{\prime}\delta
		\end{pmatrix}\begin{pmatrix}
		1\\
		k
		\end{pmatrix}$. Thus,\\ $\theta^{\prime}\theta(hk) = \begin{pmatrix}
		\alpha^{\prime}\alpha+(\gamma^{\prime}\alpha\cdot\beta^{\prime}\gamma) & \alpha^{\prime}\beta + (\gamma^{\prime}\beta \cdot \beta^{\prime}\delta)\\
		(\gamma^{\prime}\alpha)^{\beta^{\prime}\gamma}+\delta^{\prime}\gamma & (\gamma^{\prime}\beta)^{\beta^{\prime}\delta}+ \delta^{\prime}\delta
		\end{pmatrix}\begin{pmatrix}
		h\\
		k
		\end{pmatrix}$. Therefore, $T(\theta^{\prime}\theta) = \begin{pmatrix}
		\alpha^{\prime}\alpha+(\gamma^{\prime}\alpha\cdot\beta^{\prime}\gamma) & \alpha^{\prime}\beta + (\gamma^{\prime}\beta \cdot \beta^{\prime}\delta)\\
		(\gamma^{\prime}\alpha)^{\beta^{\prime}\gamma}+\delta^{\prime}\gamma & (\gamma^{\prime}\beta)^{\beta^{\prime}\delta}+ \delta^{\prime}\delta
		\end{pmatrix}=T(\theta)T(\theta^{\prime})$. Hence, $T$ is an isomorphism of groups.
	\end{proof}
	\n From here on, we will identify the automorphisms of $G$ with the matrices in $\mathcal{A}$. Let 
\begin{align*}
P = & \{\alpha\in Aut(H) \;|\; k\cdot \alpha(h) = \alpha(k\cdot h)\; \text{and}\; k^{\alpha(h)} = k^{h}\},\\
Q = & \{\beta\in Map(K,H)\;|\; \beta(kk^{\prime}) = \beta(k)(k\cdot \beta(k^{\prime})), k = k^{\beta(k^{\prime})}, \beta(k) = \beta(k^{h})\},\\
R = & \{\gamma \in Map(H,K) \;|\; \gamma(hh^{\prime}) = \gamma(h)^{h^{\prime}}\gamma(h^{\prime}), h^{\prime} = \gamma(h)\cdot h^{\prime}, \gamma(k\cdot h) = \gamma(h)\}, \\
S = & \{\delta\in Aut(K)\;|\; \delta(k)\cdot h = k\cdot h, \delta(k)^{h} = \delta(k^{h})\},\\
X =& \{(\alpha,\gamma,\delta)\in Map(H,H)\times Map(H,K)\times Aut(K)\;|\; \alpha(hh^{\prime}) = \alpha(h)(\gamma(h)\cdot \alpha(h^{\prime})),\\& \gamma(hh^{\prime}) = \gamma(h)^{\alpha(h^{\prime})}\gamma(h^{\prime}), \delta(k)\cdot \alpha(h) = \alpha(k\cdot h), \delta(k)^{\alpha(h)}\gamma(h) = \gamma(k\cdot h)\delta(k^{h}) \},\\
Y =& \{(\alpha,\beta,\delta)\in Aut(H)\times Map(K,H)\times Map(K,K)\;|\; \beta(kk^{\prime}) = \beta(k)(\delta(k)\cdot \beta(k^{\prime})),\\& \delta(kk^{\prime}) = \delta(k)^{\beta(k^{\prime})}\delta(k^{\prime}), \beta(k)(\delta(k)\cdot \alpha(h)) = \alpha(k\cdot h)\beta(k^{h}), \delta(k)^{\alpha(h)} = \delta(k^{h}) \},\\ 
Z =& \{(\alpha,\delta)\in Aut(H)\times Aut(K)\;|\; \delta(k)\cdot \alpha(h) = \alpha(k\cdot h), \delta(k)^{\alpha(h)} = \delta(k^{h})\}.	
\end{align*}
\n Then one can easily check that $P$, $S$, $X$, $Y$ and $Z$ are all subgroups of the group $Aut(G)$. But $Q$ and $R$ need not be subgroup of the group $Aut(G)$. However, if $H$ and $K$ are abelian groups, then $Q$ and $R$ are subgroups of $Aut(G)$.
\n Let
\begin{center}
	$\begin{matrix}
	A = \left\{\begin{pmatrix}
	\alpha & 0\\
	0 & 1
	\end{pmatrix}|\; \alpha\in P\right\}, & B = \left\{\begin{pmatrix}
	1 & \beta\\
	0 & 1
	\end{pmatrix}|\; \beta\in Q\right\},\\
	C = \left\{\begin{pmatrix}
	1 & 0 \\
	\gamma & 1
	\end{pmatrix}|\; \gamma\in R\right\}, & D = \left\{\begin{pmatrix}
	1 & 0\\
	0 & \delta
	\end{pmatrix}|\; \delta\in S\right\}\\
	E = \left\{\begin{pmatrix}
	\alpha & 0\\
	\gamma & \delta
	\end{pmatrix}|\; (\alpha,\gamma,\delta)\in X\right\}, & F = \left\{\begin{pmatrix}
	\alpha & \beta\\
	0 & \delta
	\end{pmatrix}|\; (\alpha,\beta,\delta)\in Y\right\},\\
	M = \left\{\begin{pmatrix}
	\alpha & 0\\
	0 & \delta
	\end{pmatrix}|\; (\alpha,\delta)\in Z\right\}.
	\end{matrix}$
\end{center}

\n be the corresponding subsets of $\mathcal{A}$. Then one can easily check that $A$, $D$, $E$, $F$ and $M$ are subgroups of $\mathcal{A}$, and if $H$ and $K$ are abelian groups, then $B$ and $C$ are also subgroups of $\mathcal{A}$.  

\begin{theorem}\label{s2t1}
Let $G=H\bowtie K$ be the Zappa-Sz\'{e}p product of two abelian groups $H$ and $K$ and $A,B,C,$ and $D$ be defined as above. Then $ABCD \subseteq \mathcal{A}$. Further, if $1-\beta\gamma \in P$, then $ABCD = \mathcal{A}$. Therefore, $Aut(G) \simeq ABCD$.
\end{theorem}
\begin{proof}
	Note that $A$ and $D$ normalizes $B$ and $C$. Then $ABCD$ is a subgroup of $Aut(G)$. Clearly, $ABCD \subseteq \mathcal{A}$. Now, let $\alpha\in P$, $\beta \in Q$, $\gamma\in R$ and $\delta \in S$. Then note that, $\alpha\beta\delta\in Q$ and $\begin{pmatrix}
		1 & \beta\\
		\gamma & 1
	\end{pmatrix} \in \mathcal{A}$. Further, let us assume that $1-\beta\gamma \in P$. Now, if $\hat{\beta} = \alpha^{-1}\beta\delta^{-1}$, then
	\[\begin{pmatrix}
	1 & \hat{\beta}\\
	\gamma & 1
	\end{pmatrix}  = \begin{pmatrix}
	1-\hat{\beta}\gamma & 0\\
	0 & 1
	\end{pmatrix} \begin{pmatrix}
	1 & (1-\hat{\beta}\gamma)^{-1}\hat{\beta}\\
	0 & 1
	\end{pmatrix} \begin{pmatrix}
	1 & 0\\
	\gamma & 1
	\end{pmatrix} \in ABC.\]
	Thus, if $\begin{pmatrix}
	\alpha & \beta\\
	\gamma & \delta
	\end{pmatrix}\in \mathcal{A}$, then 
	\[\begin{pmatrix}
	\alpha & \beta\\
	\gamma & \delta
	\end{pmatrix} = \begin{pmatrix}
	\alpha & 0\\
	0 & 1
	\end{pmatrix} \begin{pmatrix}
	1 & \hat{\beta}\\
	\gamma & 1
	\end{pmatrix} \begin{pmatrix}
	1 & 0\\
	0 & \delta
	\end{pmatrix}\in A(ABC)D = ABCD.\]
	Therefore, $\mathcal{A}\subseteq ABCD$. Hence, $ABCD = \mathcal{A}$ and so, $Aut(G)\simeq \mathcal{A}$.	
\end{proof}
	
	
	\section{Automorphisms of Zappa-Sz\'{e}p product of groups $\mathbb{Z}_{4}$ and $\mathbb{Z}_{m}$}
	
	In \cite{y4m}, Yacoub classified the groups which are Zappa-Sz\'{e}p product of cyclic groups of order $4$ and order $m$. He found that these are of the following type (see \cite[Conclusion, p. 126]{y4m})
	
	\begin{align*}
	L_1 = & \langle a,b \;|\; a^{m} = 1 = b^{4}, ab = ba^r, r^4\equiv 1 \Mod{m}\rangle, \\
	L_2 = & \langle a,b \;|\; a^{m} = 1 = b^{4}, ab = b^{3}a^{2t+1}, a^{2}b = ba^{2s}\rangle,
	\end{align*}
	
	\n where in $L_2$, $m$ is even. These are not non-isomorphic classes. The group $L_1$ may be isomorphic to the group $L_2$ depending on the values of $m,r$ and $t$ (see \cite[Theorem 5, p. 126]{y4m}). Clearly, $L_1$ is a semidirect product. Throughout this section $G$ will denote the group $L_2$ and we will be only concerned about groups $L_2$ which are Zappa-Sz\'{e}p product but not the semidirect product. Note that $G=H \bowtie K$, where $H=\langle b \rangle$ and $K=\langle a \rangle$. For the group $G$, the mutual actions of $H$ and $K$ are defined by $a\cdot b = b^{3}, a^{b} = a^{2t+1}$ along with $a^{2}\cdot b = b$ and $(a^{2})^{b} = a^{2s}$, where  $t$ and $s$ are the integers satisfying the conditions
	\begin{itemize}
		\item[$(G1)$] $2s^{2}\equiv 2 \Mod{m}$,
		\item[$(G2)$] $4t(s+1)\equiv 0 \Mod{m}$,
		\item[$(G3)$] $2(t+1)(s-1)\equiv 0 \Mod{m}$,
		\item[$(G4)$] $\gcd(s,\frac{m}{2}) = 1$.
	\end{itemize}

	\begin{lemma}\label{l1}
		\[	(a^{l})^{b} = \left\{\begin{array}{ll}
		a^{2t+1+(l-1)s}, & \text{if}\; l\; \text{is odd}\\
		a^{ls}, & \text{if}\; l\; \text{is even}
		\end{array}\right..\]
	\end{lemma}

 \begin{lemma}\label{l2}
 	Let $\begin{pmatrix}
 	\alpha & \beta\\
 	\gamma & \delta
 	\end{pmatrix}\in \mathcal{A}$. Then
	\begin{itemize}
		
		\item[$(i)$]  $Im(\delta)\subseteq \langle a^{r}\rangle$, where $r$ is odd,
		\item[$(ii)$] $\beta(a^{l}) = \left\{\begin{array}{ll}
		\beta(a),	& \text{if}\; l\; \text{is odd}\\
		1, & \text{if}\; l\; \text{is even}
		\end{array}\right.$,
		\item[$(iii)$] $Im(\gamma)\subseteq \langle a^{2}\rangle$, 
		\item[$(iv)$] $\alpha\in Aut(H)$,
		\item[$(v)$] $\beta\gamma = 0$, where $0$ is the trivial group homomorphism,
		\item[$(vi)$] $\gamma(h)\cdot \beta(k) = \beta(k)$, for all $h\in H$ and $k\in K$,
		\item[$(vii)$] If either $s=1$ or $Im(\beta)\subseteq \langle b^{2}\rangle$, then $\gamma(h)^{\beta(k)} = \gamma(h)$, for all $h\in H$ and $k\in K$.
	\end{itemize}	
\end{lemma}
\begin{proof}
	\begin{itemize}
		\item[$(i)$] 	If possible, let $\delta(a) = a^{r}$, where $r$ is even. Then, using $(A3)$ and $a^{2}\cdot b^{j} = b^{j}$, $\beta$ is a homomorphism. Also, using $(a^{2})^{b} = a^{2s}, (C4)$ and $(A4)$, if $\beta(a) = 1$ or $b^{2}$, then $\delta$ is defined by $\delta(a^{l}) = a^{rl}$, for all $l$. Similarly, if $\beta(a) = b$ or $b^{3}$, then $\delta$ is defined by
		\begin{equation*}
		\delta(a^{l}) = \left\{ \begin{array}{ll}
		a^{\frac{l+1}{2}r + \frac{l-1}{2}rs}, & \text{if}\; l\; \text{is odd}\\
		a^{\frac{l}{2}r(s+1)}, & \text{if}\; l\; \text{is even}
		\end{array} \right..
		\end{equation*}
		One can easily observe that $\delta$ is neither one-one nor onto. But this is a contradiction by $(A7)$. Hence, $Im(\delta)\subseteq \langle a^{r}\rangle$, where $r$ is odd.   
		\item[$(ii)$] Using $(C3)$, and $a\cdot b = b^{-1}$, we have if $\nu$ is odd, then $a^{\nu}\cdot b^{j} = b^{-j}$, for all $j$. Thus using $(A3)$, $(C2)$ and part $(i)$, $\beta(a^{2}) = \beta(a)(\delta(a)\cdot \beta(a)) = \beta(a)(\beta(a))^{-1} = 1$ and $\beta(a^{3}) = \beta(a)(\delta(a)\cdot \beta(a^{2})) = \beta(a)(\delta(a)\cdot 1) = \beta(a)$. Inductively, we get the required result. 
		\item[$(iii)$] Suppose that $\gamma(b) = a^{\lambda}$, where $\lambda$ is odd. Then using $(A1)$, $\alpha(b) = b^{i} = \alpha(b^{3})$ and $\alpha(b^{2}) = 1 = \alpha(1)$, where $0\le i \le 3$. Thus the map $\alpha$ is neither one-one nor onto, but by $(A7)$, the map $\alpha$ is a bijection. This is a contradiction. Thus, $\lambda$ is even. Now, using $(A2)$, for different choices of $\alpha(b)$ we find that $\gamma(b^{2})\in \{a^{2\lambda}, a^{\lambda(s+1)}\}$. Since, $\lambda$ is even, $\gamma(b^{2})\in \langle a^{2}\rangle$. Similarly, $\gamma(b^{3})\in \{a^{3\lambda}, a^{\lambda(s+2)}\}$ and so, $\gamma(b^{3})\in \langle a^{2}\rangle$. Hence, $(iii)$ holds. 
		\item[$(iv)$] Using $(iii)$ and $(A1)$, one observes that $\alpha$ is an endomorphism of $H$. Also, by $(A7)$, $\alpha$ is a bijection. Thus, $\alpha$ is an automorphism of $H$. Hence, $(iv)$ holds.
		\item[$(v)$] Using the parts $(ii)$ and $(iii)$, $\beta\gamma(h) = 1$, for all $h\in H$. Thus, $\beta\gamma = 0$.
		\item[$(vi)$] Using the relation $a^{2}\cdot b = b$ and the part $(iii)$, $(vi)$ holds.
		\item[$(vii)$] 	Using $(C4)$ and $(G1)$, we get
		\begin{equation}\label{e2}
		(a^{2l})^{b^{j}} = \left\{\begin{array}{ll}
		a^{2ls}, & \text{if}\; j\; \text{is odd}\\
		a^{2l}, & \text{if}\; j\; \text{is even}
		\end{array} \right..
		\end{equation}
		\n Thus, if either $s=1$ or $Im(\beta)\subseteq \langle b^{2}\rangle$, then using the part $(iii)$ and the Equation (\ref{e2}), $(vii)$ holds.
		
	\end{itemize}
	
\end{proof}

\n By \ref{l2} $(ii)$, observe that, $\beta(k^{h}) = \beta(k)$, for all $k\in K$ and $h\in H$.

	\begin{lemma}\label{l3}
		 Let $\beta\in Q$. Then $\beta\in Hom(K,H)$ and $Im(\beta) \le \langle b^{2}\rangle$. Also, $Im(\beta) = \langle b^{2}\rangle$ if and only if $2t(1+s)\equiv 0 \Mod{m}$, where $\gcd(s+1, \frac{m}{2})\ne 1$. 
		
	\end{lemma}
	\begin{proof}
	Let $\beta(a) = b^{i}$. Then using the Lemma \ref{l2} $(ii)$, we have, $\beta(a^{2j}) = 1$ and $\beta(a^{2j+1}) = b^{i}$, for all $j$. So, it is sufficient to study  only the $\beta(a)$ in the following,
			\begin{equation}\label{e1}
			a = a^{\beta(a)} = a^{b^{i}}.
			\end{equation}
			
		\n	Clearly, the Equation (\ref{e1}) holds trivially for $i=0$. If $i=1$, then by the Equation (\ref{e1}), $a = a^{2t+1}$ which implies that $2t\equiv 0 \Mod{m}$. Therefore, in the defining relations of the group $G$, $ab = b^{3}a$ which shows that $G$ is a semidirect product of the groups $H$ and $K$. For $i = 3$, $a = a^{b^{3}} = a^{4t+2ts+1}$, which gives that $4t+2ts \equiv 0 \Mod{m}$. So, using $(G2)$ and $(G4)$, $2ts \equiv 0 \Mod{m}$ which gives that $t\equiv 0 \Mod{\frac{m}{2}}$. Thus, $G$ is again the semidirect product of $H$ and $K$. Now, For $i=2$, using $(C6)$ and the Lemma \ref{l1}, $a^{b^{2}} = (a^{2t+1})^{b} = a^{2t+1+2ts}$. Then, $a^{b^{2}} = a$ if and only if $2t(1+s)\equiv 0 \Mod{m}$.
			
			\vspace{.2cm}
			
			\n Now, if $\gcd(s+1, \frac{m}{2}) = 1$, then $t\equiv 0 \Mod{\frac{m}{2}}$ and so, $G$ is a semidirect product of the groups $H$ and $K$. On the other hand, if $\gcd(s+1, \frac{m}{2}) \ne 1$, then $t\not\equiv 0 \Mod{\frac{m}{2}}$. Thus, $G$ is a Zappa-Sz\'{e}p product of $H$ and $K$. Thus, $Im(\beta) = \langle b^{2}\rangle$ if and only if $2t(1+s)\equiv 0 \Mod{m}$ and $\gcd(s+1, \frac{m}{2})\ne 1$. Since $Im(\beta)\subseteq \langle b^{2}\rangle$, using the Lemma \ref{l2} $(ii)$, one can easily observe that $\beta\in Hom(K,H)$. Hence, the result holds.
			
	\end{proof} 
	
%
%
%
\n Now, one can easily observe that for the given group $G$, $k\cdot \alpha(h) = \alpha(k\cdot h)$, $\beta(k) = \beta(k^{h})$, $h^{\prime} = \gamma(h)\cdot h^{\prime}$, $\delta(k)\cdot h = k\cdot h$, $\delta(k)\cdot \alpha(h) = \alpha(k\cdot h)$ and $\beta(k)(\delta(k)\cdot \alpha(h)) = \alpha(k\cdot h)\beta(k^{h})$  always holds for all $\alpha\in P$, $\beta\in Q$, $\gamma \in R$, $\delta \in S$, $(\alpha,\gamma,\delta) \in X$, $(\alpha,\delta)\in Z$ and $(\alpha, \beta, \delta) \in Y$ respectively. Thus the subgroups $P$, $Q$, $R$, $S$, $X$, $Y$ and $Z$ reduces to the following,
\begin{align*}
P = & \{\alpha\in Aut(H) \;|\; k^{\alpha(h)} = k^{h}\},\\
Q = & \{\beta\in Hom(K,H)\;|\; k = k^{\beta(k^{\prime})}\} = Hom(K,Stab_{H}(K)),\\
R = & \{\gamma \in CrossHom(H,Stab_{K}(H)) \;|\; \gamma(k\cdot h) = \gamma(h)\} \\
S = & \{\delta\in Aut(K)\;|\;\delta(k)^{h} = \delta(k^{h})\},\\
X =& \{(\alpha,\gamma,\delta)\in Aut(H)\times Map(H,K)\times Aut(K)\;|\; \gamma(hh^{\prime}) = \gamma(h)^{\alpha(h^{\prime})}\gamma(h^{\prime}),\\&  \delta(k)^{\alpha(h)}\gamma(h) = \gamma(k\cdot h)\delta(k^{h}) \},\\
Y =& \{(\alpha,\beta,\delta)\in Aut(H)\times Map(K,H)\times Map(K,K)\;|\; \beta(kk^{\prime}) = \beta(k)(\delta(k)\cdot \beta(k^{\prime})),\\& \delta(kk^{\prime}) = \delta(k)^{\beta(k^{\prime})}\delta(k^{\prime}), \delta(k)^{\alpha(h)} = \delta(k^{h}) \},\\
Z =& \{(\alpha,\delta)\in Aut(H)\times Aut(K)\;|\; \delta(k)^{\alpha(h)} = \delta(k^{h})\}.
\end{align*}
\begin{theorem}\label{s3t1}
	Let $A,B,C,$ and $D$ be defined as above. Then $Aut(G) = ABCD$. 
\end{theorem}
\begin{proof}
	Using the Lemma \ref{l2} $(v)$, we have that $\beta\gamma = 0$ and so, $1-\beta\gamma \in P$. Therefore, by the Theorem \ref{s2t1}, we have, $Aut(G) = ABCD$. 
\end{proof}
			\begin{theorem}\label{abcd}
			Let $\begin{pmatrix}
			\alpha & \beta\\
			\gamma & \delta
			\end{pmatrix} \in \mathcal{A}$. Then, if $\beta\in Q$ and $(\alpha, \gamma, \delta)\in X$, then $Aut(G) \simeq E \rtimes B  \simeq (C\rtimes M)\rtimes B$.
			\end{theorem}
	\begin{proof}
		Let $\beta\in Q$. Then using the Lemma \ref{l3}, $Im(\beta)\le \langle b^{2}\rangle$. Let $k,k^{\prime}\in K$ such that $\beta(k) = b^{2i}$ and $\beta(k^{\prime}) = b^{2j}$, for all $i,j$. Then
		\begin{align*}
		\gamma\beta(kk^{\prime}) =& \gamma(\beta(k)(k\cdot \beta(k^{\prime})))\\
		=& \gamma(\beta(k))^{\alpha(k\cdot\beta(k^{\prime}))}\gamma(k\cdot\beta(k^{\prime}))\\
		 =& \gamma(b^{2i})^{\alpha(k\cdot b^{2j})}\gamma(\beta(k^{\prime}))\\
		  =&  (a^{i\lambda(s+1)})^{\alpha(b^{2j})}\gamma(\beta(k^{\prime}))\\
		   =& (a^{i\lambda(s+1)})^{b^{2j}}\gamma(\beta(k^{\prime}))\\
		   =& a^{i\lambda(s+1)s^{2j}}\gamma(\beta(k^{\prime}))\\ 
		   =& a^{i\lambda(s+1)}\gamma(\beta(k^{\prime}))\\
		    =& \gamma(b^{2i})\gamma(\beta(k^{\prime}))\\
		     =& \gamma\beta(k)\gamma\beta(k^{\prime}).
		\end{align*}
	 Thus $\gamma\beta\in Hom(K,K)$ and so, $\gamma\beta+ \delta\in Hom(K,K)$. Now, let $\beta(a) = b^{2j}$ and $\delta(a) = a^{r}$, where $j \in \{0,1\}$ and $r\in U(m)$. Then, using the Lemma \ref{l2}, we have 
		\begin{align*}
		(\gamma\beta+\delta)(a^{l}) = \left\{\begin{array}{ll}
		a^{lr}, & \text{if}\; l\; \text{is even}\\
		a^{\lambda j(s+1)+lr}, & \text{if}\; l\; \text{is odd}
		\end{array}
		\right..
		\end{align*}
	\n	One can easily observe that $\gamma\beta+ \delta$ defined as above is a bijection. Thus $\gamma\beta+\delta \in Aut(K)$.  
		
		\vspace{.2cm}
		
		\n Now, using $(C3)$ and $(C4)$ and the Lemma \ref{l2} $(iii)$, $(\gamma\beta+\delta)(a)\cdot \alpha(b) = \gamma\beta(a)\delta(a)\cdot \alpha(b) = \gamma\beta(a)\cdot (\delta(a)\cdot \alpha(b)) = \gamma\beta(a)\cdot\alpha(a\cdot b) = \alpha(a\cdot b)$ and $(\gamma\beta+\delta)(a)^{\alpha(b)}\gamma(b) = (\gamma\beta(a)\delta(a))^{\alpha(b)}\gamma(b) = (\delta(a)\gamma\beta(a))^{\alpha(b)}\gamma(b) = \delta(a)^{\gamma(\beta(a))\cdot \alpha(b)}$  $\gamma(\beta(a))^{\alpha(b)}\gamma(b) = \delta(a)^{\alpha(b)}\gamma(b^{2i})^{\alpha(b)}\gamma(b) = \delta(a)^{\alpha(b)}\gamma(b)(a^{i\lambda(s+1)})^{\alpha(b)} = \gamma(a\cdot b)\delta(a^{b})$ $a^{i\lambda(s+1)} = \gamma(a\cdot b)\delta(a^{b})\gamma(b^{2i}) = \gamma(a\cdot b)\gamma(\beta(a))\delta(a^{b}) = \gamma(a\cdot b)\gamma(\beta(a^{2t+1}))$ $\delta(a^{b}) = \gamma(a\cdot b)\gamma(\beta(a^{b}))\delta(a^{b}) = \gamma(a\cdot b)(\gamma\beta+\delta)(a^{b})$. Thus, $(\alpha, \gamma,\gamma\beta+\delta)\in X$.
		
		\vspace{.2cm}
		
		\n Using the Lemma \ref{l2} $(v)$, we have
		\begin{equation}\label{s3e2}
		\begin{pmatrix}
		1 & \beta\\
		0 & 1
		\end{pmatrix}\begin{pmatrix}
		\alpha & 0\\
		\gamma & \delta
		\end{pmatrix}\begin{pmatrix}
		1 & \beta\\
		0 & 1
		\end{pmatrix}^{-1} = \begin{pmatrix}
		\alpha & (\alpha+\beta\gamma)(-\beta)+ \beta\delta\\
		\gamma & \gamma\beta+\delta
		\end{pmatrix}.
		\end{equation}
		\n Now, using the Lemma \ref{l2} $(ii)$, we have, $((\alpha+\beta\gamma)(-\beta)+ \beta\delta)(a) = (\alpha+\beta\gamma)(-\beta(a)) \beta(\delta(a)) = (\alpha + \beta\gamma)(b^{-2j})\beta(a^{r}) = \alpha(b^{2j})\beta(\gamma(b^{2j}))b^{2j} = b^{2ij}b^{2j} = b^{2j(i+1)} = 1$. Thus, $(\alpha+\beta\gamma)(-\beta)+ \beta\delta = 0$. Therefore, by the Equation (\ref{s3e2}),
		\[\begin{pmatrix}
		1 & \beta\\
		0 & 1
		\end{pmatrix}\begin{pmatrix}
		\alpha & 0\\
		\gamma & \delta
		\end{pmatrix}\begin{pmatrix}
		1 & \beta\\
		0 & 1
		\end{pmatrix}^{-1} = \begin{pmatrix}
		\alpha & 0\\
		\gamma & \gamma\beta+\delta
		\end{pmatrix}\in E.\] 
		So, $E \triangleleft \mathcal{A}$. Now, if $\begin{pmatrix}
		\alpha & \beta\\
		\gamma & \delta
		\end{pmatrix} \in \mathcal{A}$, then 
		\[\begin{pmatrix}
		\alpha & \beta\\
		\gamma & \delta
		\end{pmatrix}=\begin{pmatrix}
		\alpha & 0\\
		\gamma & -\gamma\alpha^{-1}\beta + \delta
		\end{pmatrix}\begin{pmatrix}
		1 & \alpha^{-1}\beta\\
		0 & 1
		\end{pmatrix}\in EB.\] 
		\n Clearly, $E\cap B = \{1\}$. Thus, $\mathcal{A}= E\rtimes B$. Hence, $Aut(G) \simeq E\rtimes B$.
		
		\vspace{.2cm}
		
		\n Let $\begin{pmatrix}
		\alpha & 0\\
		\gamma & \delta
		\end{pmatrix} \in E$. Then 
		\[\begin{pmatrix}
		\alpha & 0\\
		\gamma & \delta
		\end{pmatrix}=\begin{pmatrix}
		\alpha & 0\\
		0 &  \delta
		\end{pmatrix}\begin{pmatrix}
		1 & 0\\
		\delta^{-1}\gamma & 1
		\end{pmatrix}\in MC.\] 
		\n Clearly, $M\cap C = \{1\}$. Since $A\times D$ normalizes $C$, $C\triangleleft E$. Thus, $E = C\rtimes M$. Hence, $X \simeq C\rtimes M$ and so, $Aut(H)\simeq (C\rtimes M)\rtimes B$.
		
	\end{proof}
\n Now, we will find the structure and the order of the automorphism group $Aut(G)$. For this, we will proceed by first taking $t$ to be such that $\gcd(t,m) =1$ and then by taking $t$ such that $\gcd(t,m) = d$, where $d>1$.
	\begin{theorem}\label{t2}
		Let $4$ divides $m$ and $t$ be odd such that $\gcd(t,m) = 1$. Then 
		\begin{equation*}
		Aut(G) \simeq \left\{\begin{array}{ll}
		(\mathbb{Z}_{\frac{m}{2}}\rtimes(\mathbb{Z}_{2}\times U(m)))\rtimes \mathbb{Z}_{2}, & \text{if}\; s\in \{\frac{m}{2}-1, m-1\}\\
		\mathbb{Z}_{\frac{m}{2}}\rtimes(\mathbb{Z}_{2}\times U(m)), & \text{if}\; s\in \{\frac{m}{4}-1, \frac{3m}{4}-1\}
		\end{array} \right..
		\end{equation*}
	\end{theorem}
	\begin{proof}
		Let $\gcd(t,m) = 1$. Then, using $(G2)$, we get, $s\equiv -1 \Mod {\frac{m}{4}}$ which implies that $s\in \{\frac{m}{4} -1, \frac{m}{2} -1, \frac{3m}{4} -1, m-1\}$. Now, using $(G3)$, we get $t\equiv -1 \Mod {\frac{m}{4}}$. Then $t\in \{\frac{m}{4} -1, \frac{m}{2} -1, \frac{3m}{4} -1, m-1\}$.
		
		\vspace{.2cm}
		
		\n	Let $(\alpha, \gamma, \delta)\in X$ be such that $\alpha(b) = b^{i}$, $\gamma(b) = a^{\lambda}$ and $\delta(a) = a^{r}$, where $i \in \{1,3\}$, $\lambda$ is even, $0 \le \lambda\le m-1$, and $r\in U(m)$. Using $\gamma(hh^{\prime}) = \gamma(h)^{\alpha(h^{\prime})}\gamma(h^{\prime})$, we get $\gamma(b^{2}) = a^{\lambda(s+1)}, \gamma(b^{3}) = a^{\lambda(s+2)}$ and $\gamma(b^{4}) = 1$. We consider two cases based on the image of the map $\alpha$.
		
		\vspace{.2cm}
		
		\n \textit{Case(i)}: Let $\alpha(b) = b$. Then, using $\gamma(a\cdot b)\delta(a^{b}) = \delta(a)^{b}\gamma(b)$, $a^{\lambda(s+2)+ (2t+1)r} = \gamma(a\cdot b)\delta(a^{b}) = \delta(a)^{b}\gamma(b) = (a^{r})^{b}a^{\lambda} = a^{2t+1+(r-1)s+\lambda}$ which implies that 
		\begin{equation}\label{e3}
		\lambda(s+1)\equiv (r-1)(s-2t-1) \Mod{m}.
		\end{equation}
		\n If $s\in \{\frac{m}{2}-1, m-1\}$, then the Equation (\ref{e3}) holds for all values of $t$, $\lambda$ and $r$. Now, if $(s,t) \in \{(\frac{m}{4}-1, \frac{m}{2}-1), (\frac{m}{4}-1, m-1)\}$, then by the Equation (\ref{e3}), $r \equiv 1+\lambda \Mod{4}$. Since $\lambda$ is even, $r\equiv 1 \; \text{or}\; 3 \Mod{4}$. Again, if $(s,t) \in \{(\frac{m}{4}-1, \frac{m}{4}-1), (\frac{m}{4}-1, \frac{3m}{4}-1)\}$, then by the Equation (\ref{e3}), $r \equiv 3-\lambda \Mod{4}$. Since $\lambda$ is even, $r\equiv 1 \; \text{or}\; 3 \Mod{4}$. By the similar argument, we get the same results for $s= \frac{3m}{4}-1$. Thus, in this case, the choices for the maps $\gamma$ and $\delta$ are, $\gamma_{\lambda}(b) = a^{\lambda}$ and $\delta_{r}(a) = a^{r}$, for all $0 \le \lambda \le m-1$, $\lambda$ is even, and $r \in U(m)$.  
		
		\vspace{.2cm}
		
		\n \textit{Case(ii):} Let $\alpha(b) = b^{3}$. Then, $a^{\lambda(s+2)+ (2t+1)r} = \gamma(a\cdot b)\delta(a^{b}) = \delta(a)^{b^{3}}\gamma(b) = (a^{r})^{b^{3}}a^{\lambda} = a^{4t+2ts+1+(r-1)s+\lambda}$ which implies that 
		\begin{equation}\label{e4}
		\lambda(s+1)\equiv 2t(s+1)+(r-1)(s-2t-1) \Mod{m}.
		\end{equation}
		\n If $s\in \{\frac{m}{2}-1, m-1\}$, then the Equation (\ref{e4}) holds for all values of $t$, $\lambda$ and $r$. Now, if $(s,t) \in \{(\frac{m}{4}-1, \frac{m}{2}-1), (\frac{m}{4}-1, m-1)\}$, then by the Equation (\ref{e4}), $r \equiv 3+\lambda \Mod{4}$. Since $\lambda$ is even, $r\equiv 1 \; \text{or}\; 3 \Mod{4}$. Again, if $(s,t) \in \{(\frac{m}{4}-1, \frac{m}{4}-1), (\frac{m}{4}-1, \frac{3m}{4}-1)\}$, then by the Equation (\ref{e4}), $r \equiv 1+\lambda \Mod{4}$. Since $\lambda$ is even, $r\equiv 1 \; \text{or}\; 3 \Mod{4}$. By the similar argument, we get the same results for $s= \frac{3m}{4}-1$. Thus, in this case, the choices for the maps $\gamma$ and $\delta$ are, $\gamma_{\lambda}(b) = a^{\lambda}$ and $\delta_{r}(a) = a^{r}$, for all $0 \le \lambda \le m-1$, $\lambda$ is even, and $r \in U(m)$.  
		
		\vspace{.2cm}
		
		\n Thus combining both the \textit{cases} $(i)$ and $(ii)$, we get, for all $\alpha\in Aut(H)$, the choices for the maps $\gamma$ and $\delta$ are, $\gamma_{\lambda}(b) = a^{\lambda}$ and $\delta_{r}(a) = a^{r}$, where $0 \le \lambda \le m-1$, $\lambda$ is even, and $r \in U(m)$. So,  using the Theorem \ref{abcd}, $X \simeq \mathbb{Z}_{\frac{m}{2}}\rtimes(\mathbb{Z}_{2}\times U(m))$. Now, if $s\in \{\frac{m}{2}-1, m-1\}$, then $2t(s+1)\equiv 0 \Mod{m}$. Therefore, using the Lemma \ref{l3}, $Im(\beta) = \{b^{2}\}$ and so, $B\simeq \mathbb{Z}_{2}$. If $s\in \{\frac{m}{4}-1, \frac{3m}{4}-1\}$, then $2t(s+1)\not\equiv 0 \Mod{m}$. Therefore, using the Lemma \ref{l3}, $Im(\beta) = \{1\}$ and so, $B$ is a trivial group.  Hence, by the Theorem \ref{abcd},
		\begin{align*}
		Aut(G) \simeq E \rtimes B \simeq	\left\{\begin{array}{ll}
		(\mathbb{Z}_{\frac{m}{2}}\rtimes(\mathbb{Z}_{2}\times U(m)))\rtimes \mathbb{Z}_{2}, & \text{if}\; s\in \{\frac{m}{2}-1, m-1\}\\
		\mathbb{Z}_{\frac{m}{2}}\rtimes(\mathbb{Z}_{2}\times U(m)), & \text{if}\; s\in \{\frac{m}{4}-1, \frac{3m}{4}-1\}
		\end{array} \right..
		\end{align*}
		\end{proof}
	\begin{theorem}
		Let $m=2q$, where $q>1$ is odd and $\gcd(t,m) = 1$. Then, $Aut(G)\simeq (\mathbb{Z}_{\frac{m}{2}}\rtimes(\mathbb{Z}_{2}\times U(m)))\rtimes \mathbb{Z}_{2}$.
	\end{theorem}
			\begin{proof}
				Using $(G1), (G2),$ and $(G3)$, we get $s,t\in \{\frac{m}{2}-1, m-1\}$. Then, the result follows on the lines of the proof of the Theorem \ref{t2}.
			\end{proof}
		\begin{theorem}\label{t3}
		Let $m = 2^{n}$, $n\ge 3$. Then 
		\begin{itemize}
			\item[$(i)$] if $t$ is even, then $Aut(G) \simeq (\mathbb{Z}_{4}\rtimes(\mathbb{Z}_{2}\times (\mathbb{Z}_{2}\times\mathbb{Z}_{2^{n-2}})))\rtimes \mathbb{Z}_{2}$,
			\item[$(ii)$] if $t$ is odd, then
			\begin{align*}
			Aut(G) \simeq \left\{ \begin{array}{ll}
			(\mathbb{Z}_{2^{n-1}}\rtimes(\mathbb{Z}_{2}\times (\mathbb{Z}_{2}\times\mathbb{Z}_{2^{n-2}})))\rtimes \mathbb{Z}_{2}, & \text{if}\; s\in \{\frac{m}{2}-1, m-1\}\\
			\mathbb{Z}_{2^{n-1}}\rtimes(\mathbb{Z}_{2}\times (\mathbb{Z}_{2}\times\mathbb{Z}_{2^{n-2}})), & \text{if}\; s\in \{\frac{m}{4}-1, \frac{3m}{4}-1\}
			\end{array}\right..
			\end{align*} 
		\end{itemize}
	\end{theorem}
	\begin{proof}
		We will find the automorphism group $Aut(G)$ in two cases namely, when $t$ is even and when $t$ is odd.
		
		\vspace{.2cm}
		
		\n \textit{$Case(i)$}. Let $t$ be even. Then $2(t+1)(s-1)\equiv 0 \Mod{2^{n}}$ implies that $s\equiv 1 \Mod{2^{n-1}}$. Therefore, $s = 1, 2^{n-1} + 1$. Now, $4t(s+1)\equiv 0 \Mod{2^{n}}$ implies that $t \equiv 0 \Mod{2^{n-3}}$. Therefore, $t\in \{2^{n-3}, 2^{n-2}, 3\cdot 2^{n-3}, 2^{n-1}, 5\cdot 2^{n-3}, 3\cdot 2^{n-2}, 7\cdot 2^{n-3}, 2^{n}\}$. Note that, for $t = 2^{n-1}$ and $t=2^{n}$, $G$ is the semidirect product of $H$ and $K$. So, we consider the other values of $t$.  
		
		\vspace{.2cm}
		
		\n Let $\gamma \in R$ be such that $\gamma(b) = a^{\lambda}$, where $0\le \lambda \le m-1$ and $\lambda$ is even. Then, since, $s=1$ and $\lambda$ is even, by $(A2)$, $\gamma\in Hom(H,K)$. Now, $1 = \gamma(b^{4}) = a^{4\lambda}$ which implies that $\lambda \equiv 0 \Mod{2^{n-2}}$. Therefore, $\lambda \in \{2^{n-2}, 2^{n-1}, 3\cdot{2^{n-2}}, 2^{n}\}$. Using $\gamma(a\cdot b) = \gamma(b)$, $a^{3\lambda} = \gamma(a\cdot b) = \gamma(b) = a^{\lambda}$ implies that $\lambda \equiv 0 \Mod{2^{n-1}}$. Thus, $\lambda\in \{0, 2^{n-1}\}$ and so, $C\simeq \mathbb{Z}_{2}$.
		
		\vspace{.2cm}
		
		\n Now, let $(\alpha, \beta, \delta)\in Y$ be such that $\alpha(b) = b^{i}, \beta(a) = b^{j}$, and $\delta(a) = a^{r}$, where $ i\in \{1,3\}$, $0\le j\le 3$ and $0\le r\le 2^{n}-1$ and $r$ is odd.  Using the Lemma \ref{l2} $(ii)$, $\beta(kk^{\prime}) = \beta(k)(\delta(k)\cdot \beta(k^{\prime}))$ holds, for all $k,k^{\prime}\in K$. Now, using $\delta(kk^{\prime}) = \delta(k)^{\beta(k^{\prime})}\delta(k^{\prime})$, we get
		
		\begin{equation*}
		\delta(a^{l}) = \left\{\begin{array}{ll}
		a^{(l-1)(jt+r) + r}, & \text{if}\; l \; \text{is odd} \\
		a^{l(jt + r)}, & \text{if}\; l \; \text{is even} 
		\end{array} \right..
		\end{equation*}
		Finally, using $\delta(k^{h}) = \delta(k)^{\alpha(h)}$, $a^{2it+r} = (a^{r})^{b^{i}} = \delta(a)^{\alpha(b)} = \delta(a^{b}) = \delta(a^{2t+1}) = a^{2t(jt+r)+r}$. Thus, $2t(jt+r-i)\equiv 0 \Mod{2^{n}}$ which implies that
		\begin{equation*}
		\begin{array}{ll}
		r\equiv i \Mod{4}, & \text{if}\; t\in \{2^{n-3}, 3\cdot 2^{n-3}, 5\cdot 2^{n-3}, 7\cdot 2^{n-3}\}\; \text{and}\; n\ge 5\\
		r\equiv i+2j \Mod{4}, & \text{if}\; t\in \{2^{n-3}, 3\cdot 2^{n-3}, 5\cdot 2^{n-3}, 7\cdot 2^{n-3}\} \; \text{and}\; n=4\\
		r\equiv i \Mod{2}, & \text{if}\; t\in \{ 2^{n-2}, 3\cdot 2^{n-2}\}
		\end{array}.
		\end{equation*}
		\n Now, if $j\in \{0,2\}$, then $r \equiv i \Mod{4}$ and if $j\in \{1,3\}$, then $r \equiv i$ or $i+2 \Mod{4}$. Thus, for all $\beta \in CrossHom(K,H)$, the choices for the maps $\alpha$ and $\delta$ are, $\alpha_{i}(b) = b^{i}$ and $\delta_{r}(a) = a^{r}$, where $i\in \{1,3\}$ and $r \in U(m)$. Note that, if $\begin{pmatrix}
		\alpha & \beta\\
		0 & \delta
		\end{pmatrix}\in F$, then 
		\[\begin{pmatrix}
		\alpha & \beta\\
		0 & \delta
		\end{pmatrix} = \begin{pmatrix}
		\alpha & 0\\
		0 & \delta
		\end{pmatrix}\begin{pmatrix}
		1 & \alpha^{-1}\beta\\
		0 & 1
		\end{pmatrix}\in MB.\]
		Clearly, $M\cap B = \{1\}$ and $M$ normalizes $B$. So, $B\triangleleft F$ and $F = B\rtimes M$.	Therefore, $Y \simeq B \rtimes M \simeq \mathbb{Z}_{4}\rtimes (\mathbb{Z}_{2}  \times U(m))$. Using the Lemma \ref{l2} $(v) - (vii)$,
		\begin{equation}\label{s3e1}
		\begin{pmatrix}
		1 & 0\\
		\gamma & 1
		\end{pmatrix}\begin{pmatrix}
		\alpha & \beta \\
		0 & \delta
		\end{pmatrix}\begin{pmatrix}
		1 & 0\\
		\gamma & 1
		\end{pmatrix}^{-1} = \begin{pmatrix}
		\alpha  & \beta\\
		\gamma\alpha + (\gamma\beta + \delta)(-\gamma) & \gamma\beta+\delta
		\end{pmatrix}.
		\end{equation}
		
		\n Now, $(\gamma\alpha + (\gamma\beta + \delta)(-\gamma))(b) = \gamma\alpha(b)(\gamma\beta + \delta)(-\gamma)(b) = \gamma(b^{i})(\gamma\beta+\delta)(a^{-\lambda})$ $= a^{i\lambda}\gamma(\beta(a^{-\lambda}))\delta(a^{-\lambda}) = a^{i\lambda}\gamma(1)a^{-\lambda(jt+r)} = a^{\lambda(i-jt-r)} = 1$. Thus, $(\gamma\beta + \delta)(-\gamma) = 0$. Also, one can easily observe that $(\alpha,\beta, \gamma\beta+\delta)\in Y$. Therefore, by the Equation (\ref{s3e1}), \[\begin{pmatrix}
		1 & 0\\
		\gamma & 1
		\end{pmatrix}\begin{pmatrix}
		\alpha & \beta \\
		0 & \delta
		\end{pmatrix}\begin{pmatrix}
		1 & 0\\
		\gamma & 1
		\end{pmatrix}^{-1} = \begin{pmatrix}
		\alpha  & \beta\\
		0 & \gamma\beta+\delta
		\end{pmatrix} \in F.\]
		\n So, $F \triangleleft \mathcal{A}$. Clearly, $F \cap C = \{1\}$. Also, if $\begin{pmatrix}
		\alpha & \beta\\
		\gamma & \delta
		\end{pmatrix}\in \mathcal{A}$, then 
		\[\begin{pmatrix}
		\alpha & \beta\\
		\gamma & \delta
		\end{pmatrix} = \begin{pmatrix}
		\alpha & \beta\\
		0 & \delta
		\end{pmatrix}\begin{pmatrix}
		1 & 0\\
		\delta^{-1}\gamma & 1
		\end{pmatrix}\in FC.\]		
	\n	 Hence, $\mathcal{A} = F\rtimes C$ and so, $Aut(G) \simeq F \rtimes C\simeq (\mathbb{Z}_{4}\rtimes(\mathbb{Z}_{2}\times (\mathbb{Z}_{2}\times\mathbb{Z}_{2^{n-2}})))\rtimes \mathbb{Z}_{2}$.
		
		\vspace{.2cm}
		
		\n \textit{$Case(ii)$}. Let $t$ be odd. Then $\gcd(t,m) = 1$. Hence, the result follows from the Theorem \ref{t2}.
	\end{proof}
	\n Now, we will discuss the structure of the automorphism group $Aut(G)$ in the case when $\gcd(t,m) > 1$. 
	\begin{theorem}\label{t4}
		Let $m=4q$, where $q>1$ is odd and $\gcd(t,m) =2^{i}d$, where $i\in \{0,1,2\}$, and $d$ divides $q$. Then $Aut(G)\simeq (\mathbb{Z}_{\frac{m}{2d}}\rtimes(\mathbb{Z}_{2}\times U(m)))\rtimes \mathbb{Z}_{2}$.
	\end{theorem}
	\begin{proof}
		Let $q=du$, for some integer $u$. Then, using $(G2)$, $s\equiv -1 \Mod{u}$ which implies that $s = lu-1$, where $1\le l\le 4d$. Since, $\gcd(s,\frac{m}{2}) = 1$, $s$ is odd and so, $l$ is even. Using $(G1)$ and $(G3)$, we get $l(u\frac{l}{2}-1)\equiv 0 \Mod{d}$ and $t+1\equiv u\frac{l}{2} \Mod{q}$. Now, one can easily observe that $\gcd(l,d) = 1$ which implies that $u\frac{l}{2}-1 \equiv 0 \Mod{d}$. Thus, $2t(s+1)\equiv 2ltu\equiv 0 \Mod{m}$ and $\gcd(s+1, \frac{m}{2})\ne 1$. Therefore, using the Lemma \ref{l3}, $B\simeq \mathbb{Z}_{2}$. 
		
		\vspace{.2cm}
		
		\n Let $(\alpha,\gamma,\delta)\in X$ be such that $\alpha(b) = b^{i}$, $\gamma(b) = a^{\lambda}$ and $\delta(a) = a^{r}$, where $i\in \{1,3\}$, $0\le \lambda \le m-1$, $\lambda$ is even, and $r\in U(m)$. Then, using $\gamma(hh^{\prime}) = \gamma(h)^{\alpha(h^{\prime})}\gamma(h^{\prime})$, we have $\gamma(b^{2}) = a^{\lambda(s+1)}$, $\gamma(b^{3}) = a^{\lambda(s+2)}$, and $\gamma(b^{4}) = 1$. Now, using $\delta(a)^{\alpha(b)}\gamma(b) = \gamma(a\cdot b)\delta(a^{b})$ and the fact that $2t(s+1)\equiv 0 \Mod{m}$, $a^{\lambda(s+2)+(2t+1)r} = \gamma(b^{3})\delta(a^{2t+1}) = \gamma(a\cdot b)\delta(a^{b}) = \delta(a)^{\alpha(b)}\gamma(b) = (a^{r})^{b^{i}}a^{\lambda} = a^{2t+1+(r-1)s +\lambda + \frac{i-1}{2}2t(s+1)} = a^{2t+1+(r-1)s +\lambda}$. Thus
		\begin{equation}\label{s3e3}
		\lambda(s+1) \equiv (r-1)(s-2t-1) \Mod{m}.
		\end{equation}
		 \n Since $2t(s+1)\equiv 0 \Mod{m}$, using $(G3)$, we get $2(s-2t-1)\equiv 0 \Mod{m}$. Therefore, by the Equation (\ref{s3e3}), $\lambda lu \equiv 0 \Mod{m}$. Using the Lemma \ref{l2} $(iii)$, we get $\lambda \equiv 0 \Mod{2d}$. Thus, using the Theorem \ref{abcd}, $X\simeq \mathbb{Z}_{\frac{m}{2d}}\rtimes(\mathbb{Z}_{2}\times U(m))$. Hence, $Aut(G)\simeq E \rtimes B\simeq  (\mathbb{Z}_{\frac{m}{2d}}\rtimes(\mathbb{Z}_{2}\times U(m)))\rtimes \mathbb{Z}_{2}$. 
	\end{proof}
		
%
%

	\begin{theorem}
		Let $m=2q$, where $q>1$ is odd and $\gcd(t,m) =2^{i}d$, where $i\in \{0,1\}$, and $d$ divides $q$. Then $Aut(G)\simeq  (\mathbb{Z}_{\frac{m}{2d}}\rtimes(\mathbb{Z}_{2}\times U(m)))\rtimes \mathbb{Z}_{2}$.
	\end{theorem}
	\begin{proof}
		Follows on the lines of the proof of the Theorem \ref{t4}.
	\end{proof}
	
	\begin{theorem}
		Let $m= 2^{n}q$, $t$ be even and $\gcd(m,t) = 2^{i}d$, where $1\le i\le n$, $n\ge 3$, $q>1$ and $d$ divides $q$. Then 
		\begin{align*}
		Aut(G) \simeq \left\{\begin{array}{ll}
		(\mathbb{Z}_{4}\rtimes(\mathbb{Z}_{2}\times U(m)))\rtimes \mathbb{Z}_{2}, & \text{if}\; d=q\\
		\mathbb{Z}_{2}\times (\mathbb{Z}_{\frac{2q}{d}}\rtimes(\mathbb{Z}_{2}\times U(m))), & \text{if}\; d\ne q\; \text{and}\; n-2\le i \le n\\
		\mathbb{Z}_{\frac{4q}{d}}\rtimes(\mathbb{Z}_{2}\times U(m)), & \text{if}\; d\ne q \; \text{and}\; i = n-3
		\end{array} \right..
		\end{align*}
	\end{theorem}
	\begin{proof}
		 We consider the following four cases to find the structure of $Aut(G)$.
		
		\vspace{.2cm}
		
		\n \textit{Case($i$):} Let $d=q$ and $\gcd(t+1,m) = u$. Since, $t+1$ is odd, $u$ is odd and $u$ divides $q$. Thus, $u$ divides $t$ and so, $u = 1$. Therefore, using $(G2)$ and $(G3)$, $s\equiv 1 \Mod{\frac{m}{2}}$ and $t \equiv 0 \Mod{\frac{m}{8}}$. By the similar argument used in the proof of the Theorem \ref{t3} $(i)$, we get, $Aut(G)\simeq F\rtimes C\simeq (\mathbb{Z}_{4}\rtimes(\mathbb{Z}_{2}\times U(m)))\rtimes \mathbb{Z}_{2}$.
		
		\vspace{.2cm}
		
		\n \textit{Case($ii$):} Let $n-2\le i\le n$ and $q=du$, for some odd integer $u$. Then using $(G2)$, $s\equiv -1 \Mod{u}$ and so, $s = lu-1$, where $0\le l\le 2^{n}d$. Since, $\gcd(s,\frac{m}{2}) = 1$, $s$ is odd and so, $l$ is even. Now, using $(G1)$, $\frac{l}{2}(\frac{l}{2}u-1)\equiv 0 \Mod{2^{n-3}d}$ and by $(G3)$, $t\equiv \frac{l}{2}u-1 \Mod{2^{n-2}q}$. Since, $t$ is even, $\frac{l}{2}$ is odd and $\gcd(\frac{l}{2},d) =1$. Thus, $\frac{l}{2}u\equiv 1 \Mod{2^{n-3}d}$ and $t\equiv 2^{i}d \Mod{2^{n-2}q}$. One can easily observe that $2t(s+1)\equiv 0 \Mod{m}$. Therefore, using the similar argument as in the proof of the Theorem \ref{t2}, we get, $Aut(G)\simeq E \rtimes B\simeq  (\mathbb{Z}_{\frac{2q}{d}}\rtimes(\mathbb{Z}_{2}\times U(m)))\rtimes \mathbb{Z}_{2}$.
		
		\vspace{.2cm}
		
		\n \textit{Case($iii$):} Let $i = n-3$, $d\ne q$ and $q=du$, for some odd integer $u$. Then using $(G2)$, $s\equiv -1 \Mod{2u}$, that is, $s = 2lu-1$, where $1\le l\le 2^{n-1}d$. Now, using $(G1)$ and $(G3)$, $l(lu-1)\equiv 0 \Mod{2^{n-3}d}$ and $(t+1)(lu-1) \equiv 0 \Mod{2^{n-2}q}$. If $l$ is even, then $t \equiv lu-1 \Mod{2^{n-2}q}$ gives that $t$ is odd, which is a contradiction. Therefore, $l$ is odd. Using $(t+1)(lu-1) \equiv 0 \Mod{2^{n-2}q}$, one can easily observe that $\gcd(l,d) = 1$. Then, $lu-1 = 2^{n-3}dl^{\prime}$ and $s = 2^{n-2}dl^{\prime}+1$, where $1\le l^{\prime}\le 8u$. Clearly, $\gcd(l^{\prime},u) = 1$. Thus, $(t+1)l^{\prime} \equiv 0 \Mod{2u}$. If $l^{\prime}$ is odd, then $(t+1) \equiv 0 \Mod{2u}$ which implies that $t$ is odd. So, $l^{\prime}$ is even and so, $t = uq^{\prime} -1$, $1\le q^{\prime} < 2^{n-1}d$, $q^{\prime}$ is odd as $t$ is even. Note that $s-2t-1 = 2^{n-2}dl^{\prime} -2t = 2^{n-2}d(l^{\prime} - \frac{t}{2^{n-3}d}) =  2^{n-2}d\left(\frac{lu-1}{2^{n-3}d} - \frac{uq^{\prime}-1}{2^{n-3}d}\right) = 2^{n-2}du\left(\frac{l-q^{\prime}}{2^{n-3}d}\right)$. 
		
		\vspace{.2cm}
		
		\n Let $(\alpha,\gamma,\delta)\in X$ be such that $\alpha(b) = b^{i}$, $\gamma(b) = a^{\lambda}$ and $\delta(a) = a^{r}$, where $i\in \{1,3\}$, $0\le \lambda \le m-1$, $\lambda$ is even and $r\in U(m)$. We consider two sub-cases based on the image of the map $\alpha$.
		
		\vspace{.2cm}
		
		\n \textit{Sub-case(i):} Let $\alpha(b) = b$. Then using $\delta(a)^{\alpha(b)}\gamma(b) = \gamma(a\cdot b)\delta(a^{b})$,\\
		\n $a^{\lambda(s+2)+ (2t+1)r} = \gamma(a\cdot b)\delta(a^{b}) = \delta(a)^{b}\gamma(b) = (a^{r})^{b}a^{\lambda} = a^{2t+1+(r-1)s+\lambda}$ which implies that 
		\begin{equation*}
		\lambda(s+1)\equiv (r-1)(s-2t-1) \Mod{m}.
		\end{equation*}
		Therefore, $\lambda(2lu)\equiv 2^{n-2}du(r-1)\left(\frac{l-q^{\prime}}{2^{n-3}d}\right) \Mod{2^{n}q}$ which implies that $\lambda l\equiv 2^{n-3}d(r-1)\left(\frac{l-q^{\prime}}{2^{n-3}d}\right) \Mod{2^{n-1}d}$. Now, if $\lambda \equiv 0 \Mod{2^{n-2}d}$, then $r \equiv 1 \; \text{or}\; 3 \Mod{4}$ and vice-versa. Thus, in this sub-case, the choices for the maps $\gamma$ and $\delta$ are, $\gamma_{\lambda}(b) = a^{\lambda}$ and $\delta_{r}(a) = a^{r}$, where $\lambda$ is even and $\lambda \equiv 0 \Mod{2^{n-2}d}$, and $r\in U(m)$.
		
		\vspace{.2cm}
		
		\n \textit{Sub-case(ii):} Let $\alpha(b) = b^{3}$. Then using $\delta(a)^{\alpha(b)}\gamma(b) = \gamma(a\cdot b)\delta(a^{b})$,\\
		\n $a^{\lambda(s+2)+ (2t+1)r} = \gamma(a\cdot b)\delta(a^{b}) = \delta(a)^{b^{3}}\gamma(b) = (a^{r})^{b^{3}}a^{\lambda} = a^{4t+2ts+1+(r-1)s+\lambda}$ which implies that 
		\begin{equation*}
		(\lambda-2t)(s+1)\equiv (r-1)(s-2t-1) \Mod{m}.
		\end{equation*}
		Therefore, $2lu(\lambda-2t)\equiv 2^{n-2}du(r-1)\left(\frac{l-q^{\prime}}{2^{n-3}d}\right) \Mod{2^{n}q}$ which implies that $l(\lambda-2t)\equiv 2^{n-3}d(r-1)\left(\frac{l-q^{\prime}}{2^{n-3}d}\right) \Mod{2^{n-1}d}$. Now, if $\lambda \equiv 0 \Mod{2^{n-2}d}$, then $r \equiv 1 \; \text{or}\; 3 \Mod{4}$ and vice-versa. Thus, in this sub-case, the choices for the maps $\gamma$ and $\delta$ are, $\gamma_{\lambda}(b) = b^{\lambda}$ and $\delta_{r}(a) = a^{r}$, where $\lambda$ is even and $\lambda \equiv 0 \Mod{2^{n-2}d}$, and $r\in U(m)$.
		
		\vspace{.2cm}
		
		\n Thus combining both the \textit{sub-cases} $(i)$ and $(ii)$, we get, for all $\alpha\in Aut(H)$, the choices for the maps $\gamma$ and $\delta$ are, $\gamma_{\lambda}(b) = a^{\lambda}$ and $\delta_{r}(a) = a^{r}$,  where $\lambda$ is even and $\lambda \equiv 0 \Mod{2^{n-2}d}$, and $r\in U(m)$. Therefore, using the Theorem \ref{abcd}, $X \simeq \mathbb{Z}_{4\frac{q}{d}}\rtimes(\mathbb{Z}_{2}\times U(m))$. At last, since, $l$ is odd, $2t(s+1) \equiv 4tlu\not\equiv 0 \Mod{m}$. Therefore, using the Lemma \ref{l3}, $Im(\beta) = \{1\}$. Thus, $B$ is a trivial group. Hence, using the Theorem \ref{abcd}, $Aut(G)\simeq E\rtimes B \simeq \mathbb{Z}_{\frac{4q}{d}}\rtimes(\mathbb{Z}_{2}\times U(m))$.
		
		\vspace{.2cm}
		
		\n \textit{Case($iv$):} Let $1\le i \le n-4$. and $q=du$, for some odd integer $u$. Then using $(G2)$, $s\equiv -1 \Mod{2^{n-i-2}u}$, that is, $s = 2^{n-i-2}lu-1$, where $1\le l\le 2^{i+2}d$. Now, using $(G1)$ and $(G3)$, $l(2^{n-i-3}lu-1)\equiv 0 \Mod{2^{i}d}$ and $(t+1)(lu2^{n-i-3}-1) \equiv 0 \Mod{2^{n-2}q}$. Since, $n-i-3>0$, $lu2^{n-i-3}-1$ is odd. If $l$ is even, then $t \equiv lu2^{n-i-3}-1 \Mod{2^{n-2}q}$ gives that $t$ is odd, which is a contradiction. Now, if $l$ is odd, then Using $(t+1)(lu-1) \equiv 0 \Mod{2^{n-2}q}$, one can easily observe that $\gcd(l,d) = 1$. Thus, $2^{n-i-3}lu-1\equiv 0 \Mod{2^{i}d}$, which is absurd. Hence, there is no such $l$ exist and so, no such $t$ and $s$ exist and hence no group $G$ exists as the Zappa-Sz\'{e}p product of $H$ and $K$.	
	\end{proof}
	
	\begin{theorem}
		Let $m = 2^{n}q$, $t$ be odd and $\gcd(t,m) = d$, where $n\ge 4$ and $q$ is odd. Then 
		\begin{equation*}
		Aut(G) \simeq \left\{\begin{array}{ll}
		(\mathbb{Z}_{\frac{m}{2d}}\rtimes(\mathbb{Z}_{2}\times U(m)))\rtimes \mathbb{Z}_{2}, & \text{if}\; 2t(s+1)\equiv 0 \Mod{m}\\
		\mathbb{Z}_{\frac{m}{2d}}\rtimes(\mathbb{Z}_{2}\times U(m)), & \text{if}\; 2t(s+1)\not\equiv 0 \Mod{m}
		\end{array}	\right..
		\end{equation*}
	\end{theorem}
	\begin{proof}
		Let $q=du$, for some odd integer $u$. Then using $(G2)$, we have $s\equiv -1 \Mod{2^{n-2}u}$ which implies that $s = 2^{n-2}lu-1$, where $1\le l \le 4d$. Now, using $(G1)$, $l(2^{n-3}ul-1)\equiv 0 \Mod{d}$. Using $(G3)$, we get
		\begin{equation}\label{e5}
		(t+1)(lu2^{n-3}-1)\equiv 0 \Mod{2^{n-2}q}.
		\end{equation}
		\n \textit{Case($i$):} If $l$ is even, then by the Equation (\ref{e5}), $t\equiv lu2^{n-3}-1 \Mod{2^{n-2}q}$. Note that, $2t(s+1)\equiv 2t(2^{n-2}lu)\equiv 0 \Mod{m}$ and $\lambda(s+1) = \lambda(lu2^{n-2})$. Thus $\lambda(s+1)\equiv 0 \Mod{m}$ if and only if $\lambda l \equiv 0 \Mod{4d}$, which is true for all $\lambda\equiv 0 \Mod{2d}$. Using the similar argument as in the proof of the Theorem \ref{t2}, we get $X \simeq \mathbb{Z}_{\frac{m}{2d}}\rtimes(\mathbb{Z}_{2}\times U(m))$ and $B\simeq \mathbb{Z}_{2}$. Hence, $Aut(G)\simeq E \rtimes B\simeq (\mathbb{Z}_{\frac{m}{2d}}\rtimes(\mathbb{Z}_{2}\times U(m)))\rtimes \mathbb{Z}_{2}$.
		
		\vspace{.2cm}
		
		\n \textit{Case($ii$):} If $l$ is odd, then using the Equation (\ref{e5}), one can easily observe that $\gcd(l,d) = 1$ which means that $2^{n-3}lu-1 = dl^{\prime}$, where $l^{\prime}$ is odd, $\gcd(l^{\prime},u) = 1$ and $1\le l^{\prime}\le 2^{n}u$. Thus, using the Equation (\ref{e5}), $(t+1)dl^{\prime} \equiv 0 \Mod{2^{n-2}q}$. Since, $\gcd(l^{\prime},u) = 1$, $t = 2^{n-2}uq^{\prime}-1$, where $1\le q^{\prime}\le 4d$. Now, $s-2t-1 = 2dl^{\prime}-2t = 2d(l^{\prime} - \frac{t}{d}) = 2d(\frac{2^{n-3}ul-2^{n-2}uq^{\prime}}{d}) = 2^{n-2}du\frac{l-2q^{\prime}}{d}$.
		
		\vspace{.2cm}
		
		\n Let $(\alpha,\gamma,\delta)\in X$ be such that $\alpha(b) = b^{i}$, $\gamma(b) = a^{\lambda}$ and $\delta(a) = a^{r}$, where $i\in \{1,3\}$, $0\le \lambda \le m-1$, $\lambda$ is even and $r\in U(m)$. We consider two sub-cases based on the image of the map $\alpha$.
		
		\vspace{.2cm}
		
		\n \textit{Sub-case ($i$):} Let $\alpha(b) = b$. Then using $\delta(a)^{\alpha(b)}\gamma(b) = \gamma(a\cdot b)\delta(a^{b})$, we get
		\n $a^{\lambda(s+2)+ (2t+1)r} = \gamma(a\cdot b)\delta(a^{b}) = \delta(a)^{b}\gamma(b) = (a^{r})^{b}a^{\lambda} = a^{2t+1+(r-1)s+\lambda}$ which implies that 
		\begin{equation*}
		\lambda(s+1)\equiv (r-1)(s-2t-1) \Mod{m}.
		\end{equation*}
		\n	Therefore, $\lambda(lu2^{n-2})\equiv 2^{n-2}q(r-1)(\frac{l-2q^{\prime}}{d}) \Mod{2^{n}q}$ which implies that $\lambda l\equiv d(r-1)(\frac{l-2q^{\prime}}{d}) \Mod{4d}$. Now, if $\lambda \equiv 0 \Mod{2d}$, then $r \equiv 3 \Mod{4}$. Again, if $\lambda \equiv 0 \Mod{4d}$, then $r \equiv 1 \Mod{4}$. Thus, in this sub-case, the choices for the maps $\gamma$ and $\delta$ are, $\gamma_{\lambda}(b) = a^{\lambda}$ and $\delta_{r}(a) = a^{r}$, where $\lambda$ is even and $\lambda \equiv 0 \Mod{2d}$, and $r\in U(m)$.
		
		\vspace{.2cm}
		
		\n \textit{Sub-case ($ii$):} Let $\alpha(b) = b^{3}$. Then using $\delta(a)^{\alpha(b)}\gamma(b) = \gamma(a\cdot b)\delta(a^{b})$, $a^{\lambda(s+2)+ (2t+1)r} = \gamma(a\cdot b)\delta(a^{b}) = \delta(a)^{b^{3}}\gamma(b) = (a^{r})^{b^{3}}a^{\lambda} = a^{4t+2ts+1+(r-1)s+\lambda}$ which implies that 
		\begin{equation*}
		(\lambda-2t)(s+1)\equiv (r-1)(s-2t-1) \Mod{m}.
		\end{equation*}
		\n	Therefore, $lu2^{n-2}(\lambda-2t)\equiv 2^{n-2}q(r-1)(\frac{l-2q^{\prime}}{d}) \Mod{2^{n}q}$ which implies that $l(\lambda-2t)\equiv d(r-1)(\frac{l-2q^{\prime}}{d}) \Mod{4d}$. Now, if $\lambda \equiv 0 \Mod{2d}$, then $r \equiv 1 \Mod{4}$. Again, if $\lambda \equiv 0 \Mod{4d}$, then $r \equiv 3 \Mod{4}$. Thus, in this sub-case, the choices for the maps $\gamma$ and $\delta$ are, $\gamma_{\lambda}(b) = a^{\lambda}$ and $\delta_{r}(a) = a^{r}$, where $\lambda$ is even and $\lambda \equiv 0 \Mod{2d}$, and $r\in U(m)$.
		
		\vspace{.2cm}
		
		\n Thus combining both the \textit{sub-cases} $(i)$ and $(ii)$, we get, for all $\alpha\in Aut(H)$, the choices for the maps $\gamma$ and $\delta$ are, $\gamma_{\lambda}(b) = a^{\lambda}$ and $\delta_{r}(a) = a^{r}$,  where $\lambda$ is even and $\lambda \equiv 0 \Mod{2d}$, and $r\in U(m)$. Therefore, using the Theorem \ref{abcd}, $E \simeq \mathbb{Z}_{\frac{m}{2d}}\rtimes(\mathbb{Z}_{2}\times U(m))$. Also, since, $2t(s+1)\not\equiv 0 \Mod{m}$, using the Lemma \ref{l3}, $Im(\beta) = \{1\}$. Thus, $B$ is a trivial group. Hence, using the Theorem \ref{abcd}, $Aut(G) \simeq E\rtimes B \simeq \mathbb{Z}_{\frac{m}{2d}}\rtimes(\mathbb{Z}_{2}\times U(m))$.
	\end{proof}
	
	\begin{theorem}
		Let $m=8q$, $t$ is odd, and $\gcd(t,m) = d$, where $q>1$ is odd. Then
		\begin{equation*}
		Aut(G) \simeq \left\{\begin{array}{ll}
		(\mathbb{Z}_{\frac{m}{2d}}\rtimes(\mathbb{Z}_{2}\times U(m)))\rtimes \mathbb{Z}_{2}, & \text{if}\; 2t(s+1)\equiv 0 \Mod{m}\\
		\mathbb{Z}_{\frac{m}{2d}}\rtimes(\mathbb{Z}_{2}\times U(m)), & \text{if}\; 2t(s+1)\not\equiv 0 \Mod{m}
		\end{array}	\right.
		\end{equation*}
	\end{theorem}
	\begin{proof}
		Let $q=du$, for some odd integer $u$. Then using $(G2)$, $s\equiv -1 \Mod{2u}$ which implies that $s = 2lu-1$, where $1\le l \le 4d$. Now, using $(G1)$, $l(lu-1)\equiv 0 \Mod{d}$. Using $(G3)$, we get
		\begin{equation}\label{e6}
		(t+1)(lu-1)\equiv 0 \Mod{2q}.
		\end{equation}
		\n \textit{Case($i$):} If $l$ is even, then by the Equation (\ref{e6}), $t\equiv lu-1 \Mod{2q}$. Note that, $2t(s+1)\equiv 2t(2lu)\equiv 0 \Mod{m}$ and $\lambda(s+1)= \lambda(2lu)$. Thus $\lambda(s+1)\equiv 0 \Mod{m}$ if and only if $\lambda l \equiv 0 \Mod{4d}$ which is true for all $\lambda\equiv 0 \Mod{2d}$. Therefore, using the similar argument as in the proof of the Theorem \ref{t2}, we get $E\simeq \mathbb{Z}_{\frac{m}{2d}}\rtimes(\mathbb{Z}_{2}\times U(m))$ and $B\simeq \mathbb{Z}_{2}$. Hence, by the Theorem \ref{abcd}, $Aut(G) \simeq E\rtimes B\simeq (\mathbb{Z}_{\frac{m}{2d}}\rtimes(\mathbb{Z}_{2}\times U(m)))\rtimes \mathbb{Z}_{2}$.
		
		\vspace{.2cm}
		
		\n \textit{Case($ii$):} If $l$ is odd, then using the Equation (\ref{e6}), one can easily observe that $\gcd(l,d) = 1$ which means that $lu-1 = dl^{\prime}$, where $1\le l^{\prime}\le 8u$ and $\gcd(l^\prime, u) = 1$. Since $lu-1$ is even, $l^{\prime}$ is even. Thus using the Equation (\ref{e6}), $(t+1)dl^{\prime} \equiv 0 \Mod{2q}$. Since, $\gcd(l^{\prime},u) = 1$, $t = uq^{\prime}-1$, where $1\le q^{\prime}\le 8d$ and $q^{\prime}$ is even, as $t$ is odd. Now, $s-2t-1 = 2dl^{\prime}-2t = 2d(l^{\prime} - \frac{t}{d}) = 2d(\frac{ul-uq^{\prime}}{d}) = 2du\frac{l-q^{\prime}}{d}$.
		
		\vspace{.2cm}
		
		\n Let $(\alpha,\gamma,\delta)\in X$ be such that $\alpha(b) = b^{i}$, $\gamma(b) = a^{\lambda}$ and $\delta(a) = a^{r}$, where $i\in \{1,3\}$, $0\le \lambda \le m-1$, $\lambda$ is even and $r\in U(m)$. We consider two sub-cases based on the image of the map $\alpha$.
		
		\vspace{.2cm}
		
		\n \textit{Sub-case($i$):} Let $\alpha(b) = b$. Then, $a^{\lambda(s+2)+ (2t+1)r} = \gamma(a\cdot b)\delta(a^{b}) = \delta(a)^{b}\gamma(b) = (a^{r})^{b}a^{\lambda} = a^{2t+1+(r-1)s+\lambda}$ which implies that 
		\begin{equation*}
		\lambda(s+1)\equiv (r-1)(s-2t-1) \Mod{m}.
		\end{equation*}
		Therefore, $\lambda(2lu)\equiv 2du(r-1)(\frac{l-q^{\prime}}{d}) \Mod{8q}$ which implies that $\lambda(l)\equiv d(r-1)(\frac{l-q^{\prime}}{d}) \Mod{4d}$. Now, if $\lambda \equiv 0 \Mod{2d}$, then $r \equiv 3 \Mod{4}$. Again, if $\lambda \equiv 0 \Mod{4d}$, then $r \equiv 1 \Mod{4}$. Thus, in this sub-case, the choices for the maps $\gamma$ and $\delta$ are, $\gamma_{\lambda}(b) = a^{\lambda}$ and $\delta_{r}(a) = a^{r}$, where $\lambda$ is even and $\lambda \equiv 0 \Mod{2d}$, and $r\in U(m)$. 
		
		\vspace{.2cm}
		
		\n \textit{Sub-case($ii$):} Let $\alpha(b) = b^{3}$. Then, $a^{\lambda(s+2)+ (2t+1)r} = \gamma(a\cdot b)\delta(a^{b}) = \delta(a)^{b^{3}}\gamma(b)$ $= (a^{r})^{b^{3}}a^{\lambda} = a^{4t+2ts+1+(r-1)s+\lambda}$ which implies that 
		\begin{equation*}
		(\lambda-2t)(s+1)\equiv (r-1)(s-2t-1) \Mod{m}.
		\end{equation*}
		Therefore, $2lu(\lambda-2t)\equiv 2du(r-1)(\frac{l-q^{\prime}}{d}) \Mod{8q}$ which implies that $l(\lambda-2t)\equiv d(r-1)(\frac{l-q^{\prime}}{d}) \Mod{4d}$. Now, if $\lambda \equiv 0 \Mod{2d}$, then $r \equiv 1 \Mod{4}$. Again, if $\lambda \equiv 0 \Mod{4d}$, then $r \equiv 3 \Mod{4}$. Thus, in this sub-case, the choices for the maps $\gamma$ and $\delta$ are, $\gamma_{\lambda}(b) = a^{\lambda}$ and $\delta_{r}(a) = a^{r}$, where $\lambda$ is even and $\lambda \equiv 0 \Mod{2d}$, and $r\in U(m)$.
		
		\vspace{.2cm}
		
		\n Thus combining both the \textit{sub-cases} $(i)$ and $(ii)$, we get, for all $\alpha\in Aut(H)$, the choices for the maps $\gamma$ and $\delta$ are, $\gamma_{\lambda}(b) = a^{\lambda}$ and $\delta_{r}(a) = a^{r}$,  where $\lambda$ is even and $\lambda \equiv 0 \Mod{2d}$, and $r\in U(m)$. Therefore, using the Theorem \ref{abcd}, $X \simeq \mathbb{Z}_{\frac{m}{2d}}\rtimes(\mathbb{Z}_{2}\times U(m))$. Also, since, $2t(s+1)\not\equiv 0 \Mod{m}$, using the Lemma \ref{l3}, $Im(\beta) = \{1\}$. Thus, $B$ is a trivial group. Hence, by the Theorem \ref{abcd}, $Aut(G)\simeq E\rtimes B \simeq \mathbb{Z}_{\frac{m}{2d}}\rtimes(\mathbb{Z}_{2}\times U(m))$.
		
	\end{proof}

\section{Automorphisms of Zappa-Sz\'{e}p product of groups $\mathbb{Z}_{p^{2}}$ and $\mathbb{Z}_{m}$, $p$ is odd prime}

In \cite{ypm}, Yacoub classified the groups which are Zappa-Sz\'{e}p product of cyclic groups of order $p^{2}$ and order $m$. He found that these are of the following type (see \cite[Conclusion, p. 38]{ypm})

\begin{align*}
M_1 = & \langle a,b \;|\; a^{m} = 1 = b^{p^{2}}, ab = ba^u, u^{p^{2}}\equiv 1 \Mod{m}\rangle, \\
M_2 = & \langle a,b \;|\; a^{m} = 1 = b^{p^{2}}, ab = b^{t}a, t^{m}\equiv 1 \Mod{p^{2}}\rangle, \\
M_3 = & \langle a,b \;|\; a^{m} = 1 = b^{p^{2}}, ab = b^{t}a^{pr+1}, a^{p}b = ba^{p(pr+1)} \rangle,
\end{align*}

\n where $p$ is an odd prime and in $M_3$, $p$ divides $m$. These are not non isomorphic classes. The groups $M_1$ and $M_2$ may be isomorphic to the group $M_3$ depending on the values of $m,r$ and $t$. Clearly, $M_1$ and $M_2$ are semidirect products. Throughout this section $G$ will denote the group $M_3$ and we will be only concerned about groups $M_3$ which are Zappa-Sz\'{e}p product but not the semidirect product. Note that $G=H \bowtie K$, where $H=\langle b \rangle$ and $K=\langle a \rangle$. For the group $G$, the mutual actions of $H$ and $K$ are defined by $a\cdot b = b^{t}, a^{b} = a^{pr+1}$ along with $a^{p}\cdot b = b$ and $(a^{p})^{b} = a^{p(pr+1)}$, where  $t$ and $r$ are integers satisfying the conditions
\begin{itemize}
	\item[$(G1)$] $\gcd(t-1, p^{2}) = p$, that is, $t = 1+\lambda p$, where $\gcd(\lambda, p) = 1$,
	\item[$(G2)$] $\gcd(r,p) = 1$,
	\item[$(G3)$] $p(pr+1)^{p}\equiv p \Mod{m}$.
\end{itemize}

\begin{lemma}\label{s4l1}
	$a^{(pr+1)^{ip\lambda}} = a^{i((pr+1)^{p\lambda} - 1)+1}$, for all $i$.
\end{lemma}
\begin{proof}
	One can easily prove the result using $(G3)$.
\end{proof}
\begin{lemma}\label{s4l2}
\begin{itemize}
	\item[$(i)$] $a\cdot b^{j} = b^{jt}$, for all $j$,
	\item[$(ii)$] $a^{l}\cdot b = b^{1+lp\lambda}$, for all $l$,
	\item[$(iii)$] $a^{(b^{j})} = a^{(pr+1)^{j}}$, for all $j$,
	\item[$(iv)$] $(a^{l})^{b} = a^{\frac{l(l-1)}{2}((pr+1)^{\lambda p}-1)+l(pr+1)}$, for all $l$,
	\item[$(v)$] $a^{l}\cdot b^{j} = b^{jt^{l}}$, for all $j$ and $l$,
    \item[$(vi)$] $(a^{l})^{b^{j}} = a^{\frac{jl(l-1)}{2}((pr+1)^{\lambda p}-1)+l(pr+1)^{j}}$, for all $j$ and $l$.
\end{itemize}
\end{lemma}
\begin{proof}
	\begin{itemize}
		\item[$(i)$] Using $(C3)$ and $(C5)$, $a\cdot b^{2} = (a\cdot b)(a^{b}\cdot b) = b^{t}(a^{pr+1}\cdot b) = b^{t}(a\cdot (a^{pr}\cdot b)) = b^{t}(a\cdot b) = b^{2t}$. Similarly, $a\cdot b^{3} = (a\cdot b)(a^{b}\cdot b^{2}) = b^{t}(a^{pr+1}\cdot b^{2}) = b^{t}(a\cdot (a^{pr}\cdot b^{2})) = b^{t}(a\cdot b^{2}) = b^{3t}$. Inductively, we get $a\cdot b^{j} = b^{jt}$, for all $j$.
		\item[$(ii)$] Using $(C3)$ and part $(i)$, $a^{2}\cdot b = a\cdot (a\cdot b) = a\cdot b^{t} = b^{t^{2}} = b^{1+2p\lambda}$. Similarly, $a^{3}\cdot b = a\cdot (a^{2}\cdot b) = a\cdot b^{t^{2}} = b^{t^{3}} = b^{1+3p\lambda}$. Inductively, we get, $a^{l}\cdot b = b^{1+lp\lambda}$, for all $l$.
		\item[$(iii)$]  First, note that, using $(C4)$, we have $(a^{lp})^{b} = a^{lp(pr+1)}$. Now, using $(C4)$ and $(C6)$, $a^{(b^{2})} = (a^{b})^{b} = (a^{pr+1})^{b} = a^{(a^{pr}\cdot b)}(a^{pr})^{b} = a^{b}a^{pr(pr+1)} = a^{(pr+1)^{2}}$. Similarly, $a^{(b^{3})} = (a^{b})^{b^{2}} = (a^{pr+1})^{b^{2}} = a^{(a^{pr}\cdot b^{2})}(a^{pr})^{b^{2}} = a^{b^{2}}$ $((a^{pr})^{b})^{b} = a^{(pr+1)^{2}}(a^{pr(pr+1)})^{b} = a^{(pr+1)^{2}}a^{pr(pr+1)^{2}} = a^{(pr+1)^{3}}$. Inductively, we get, $a^{(b^{j})} = a^{(pr+1)^{j}}$, for all $j$. 
		\item[$(iv)$] Using $(C4)$, $(G3)$ and the part $(iii)$, $(a^{2})^{b} = a^{(a\cdot b)}a^{b} = a^{(b^{t})}a^{pr+1} = a^{(pr+1)^{(1+\lambda p)}}$ $a^{pr+1} = a^{(pr+1)^{\lambda p}+ pr(pr+1)^{\lambda p} + pr+1} = a^{((pr+1)^{\lambda p} -1)+ 2(pr+1)}$. By the similar argument, we get,
		\begin{align*}
		(a^{3})^{b} =& (a^{2})^{(a\cdot b)}a^{b}\\
		=& (a^{2})^{b^{t}}a^{pr+1}\\
		=& a^{(a\cdot b^{t})}a^{(b^{t})}a^{pr+1}\\
		=& a^{(b^{1+2p\lambda})}a^{(b^{1+\lambda p})}a^{pr+1}\\
		=& a^{(pr+1)^{1+2p\lambda}+(pr+1)^{1+p\lambda}+ pr+1}\\
		=& a^{(pr+1)^{2p\lambda}+pr(pr+1)^{2p\lambda}+ (pr+1)^{p\lambda} + pr(pr+1)^{p\lambda}+ pr+1}\\
		=& a^{2((pr+1)^{p\lambda}-1) + 1 + pr + (pr+1)^{p\lambda} + pr + pr+1},\;\text{(using the Lemma \ref{s4l1})}\\
		=& a^{3((pr+1)^{p\lambda} -1)+ 3(pr+1)}.
		\end{align*}
	\n	Inductively, we get, $(iv)$.
		\item[$(v)$] Follows inductively, using the parts $(i)$ and $(ii)$.
		\item[$(vi)$] Follows inductively, using the parts $(iii)$ and $(iv)$.
	\end{itemize}
\end{proof}

\begin{lemma}\label{s4l3}
	If for all $l\ne 0$, $(pr+1)^{pl}\not\equiv 1 \Mod{m}$, then 
	\begin{itemize}
		\item[$(i)$] $Im(\gamma)\subseteq \langle a^{p}\rangle$,
		\item[$(ii)$] $\alpha\in Aut(H)$,
	\end{itemize} 
\end{lemma}
\begin{proof}
	\begin{itemize}
		\item[$(i)$] Let $\alpha(b) = b^{i}$ and $\gamma(b) = a^{\mu}$. Then using $(A1)$ and the Lemma \ref{s4l2} $(v)$, $\alpha(b^{2}) = \alpha(b)(\gamma(b)\cdot \alpha(b)) = b^{i}(a^{\mu} \cdot b^{i}) = b^{i(1+t^{\mu})}$. Inductively, we get,
		\begin{align*}
		\alpha(b^{u}) &= b^{i(1+ t^{\mu} + t^{2\mu} + \cdots + t^{(u-1)\mu})}\\
		&= b^{i(1+ (1+p\mu\lambda) + (1+ 2p\mu\lambda) + \cdots + (1+ (u-1)p\mu\lambda))}\\
		&= b^{i(u+ \frac{u(u-1)}{2}p\mu\lambda)}
		\end{align*}  
		\n for all $0\le u \le p^{2}-1$. Now, using $(A2)$ and the Lemma \ref{s4l2} $(vi)$, $\gamma(b^{2}) = \gamma(b)^{\alpha(b)}\gamma(b) = (a^{\mu})^{b^{i}}a^{\mu} = a^{\frac{i\mu(\mu-1)}{2}((pr+1)^{p\lambda} - 1)+ \mu(pr+1)^{i}+ \mu}$. Inductively, we get, 
		\begin{align*}
		\gamma(b^{u}) = a^{(i\frac{u(u-1)\mu(\mu-1)}{2}+ i\mu^{2}\frac{u(u-1)(u-2)}{6})((pr+1)^{p\lambda} - 1)+\mu \sum_{\nu=0}^{u-1}{(pr+1)^{i\nu}}}
		\end{align*}
		\n for all $0\le u \le p^{2}-1$. Now, using $(G3)$, $1 = \gamma(p^{2}) = a^{\mu\sum_{\nu=0}^{p^{2}-1}{(pr+1)^{i\nu}}} = a^{\mu\left(\frac{(pr+1)^{ip^{2}}-1}{(pr+1)^{i}-1}\right)}$ which implies that
		\begin{equation}\label{s4e1}
		\mu\left(\frac{(pr+1)^{ip^{2}}-1}{(pr+1)^{i}-1}\right)\equiv 0 \Mod{m}.
		\end{equation}
		\n If for all $l\ne 0$, $(pr+1)^{pl}\equiv 1 \Mod{m}$, then by the Equation (\ref{s4e1}), $\mu$ can be anything. On the other hand, if for all $l\ne 0$, $(pr+1)^{pl}\not\equiv 1 \Mod{m}$, then by the Equation (\ref{s4e1}) and $(G3)$, $\mu \equiv 0 \Mod{p}$. Also, note that, in both the cases, namely  $(pr+1)^{pl}\equiv 1 \Mod{m}$ and $(pr+1)^{pl}\not\equiv 1 \Mod{m}$, we have that $\gamma(b^{u}) = a^{\mu \sum_{\nu=0}^{u-1}{(pr+1)^{i\nu}}}$. Hence, if $(pr+1)^{pl}\not\equiv 1 \Mod{m}$, then  $\gamma(b^{u}) = a^{\mu \sum_{\nu=0}^{u-1}{(pr+1)^{i\nu}}}\in \langle a^{p}\rangle$.  
		\item[$(ii)$] Follows immediately using the part $(i)$.
	\end{itemize}
\end{proof}

\begin{lemma}\label{s4l4}
	Let $\begin{pmatrix}
		\alpha & \beta\\
		\gamma & \delta
	\end{pmatrix} \in \mathcal{A}$. Then, if $\beta\in Q$, then 
	\begin{itemize}
		\item[$(i)$] $\beta\in Hom(K,H)$ and $Im(\beta)\le \langle b^{p}\rangle$,
		\item[$(ii)$] $l(pr+1)^{j}\equiv l \Mod{m}$, for all $l$,
		\item[$(iii)$] $\gamma(h)\cdot \beta(k) = \beta(k)$ and $\gamma(h)^{\beta(k)} = \gamma(h)$, for all $h\in H$ and $k\in K$,
		\item[$(iv)$] $\gamma\beta = 0$, where $0$ is the trivial homomorphism in $Hom(K,K)$,
		\item[$(v)$] $\gamma\beta + \delta \in Aut(K)$ and $\gamma\beta + \delta\in S$, 
		\item[$(vi)$] $\beta\gamma \in Hom(H,H)$,
		\item[$(vii)$]  $\alpha+\beta\gamma \in Aut(H)$ and $\alpha+\beta\gamma \in P$.
	\end{itemize}
	
\end{lemma}
\begin{proof}
	 Let $\beta(a) = b^{j}$. Then using $(A3)$, $\beta(a^{2}) = \beta(a)(a\cdot \beta(a)) = b^{j}(a\cdot b^{j}) = b^{j+jt}$. Inductively, we get, 
	 \[\beta(a^{l}) = b^{j(1+t+t^{2}+ \cdots + t^{l-1})} = b^{j(1 + (1+\lambda p) + (1+ 2\lambda p)+ \cdots + (1+(l-1)\lambda p))} = b^{j(l+\lambda p\frac{l(l-1)}{2})}.\] 
	 \begin{itemize}
	 	\item[$(i)$] Since $\beta\in Q$, $\beta(k^{h}) = \beta(k)$. Therefore, $b^{j} = \beta(a) = \beta(a^{b}) = \beta(a^{pr+1}) = b^{j(pr+1)}$ which implies that $jpr + j \equiv j \Mod{p^{2}}$. Since $\gcd(r,p) = 1$, $j\equiv 0 \Mod{p}$. Thus, $\beta(a^{l}) = b^{jl}\in \langle b^{p}\rangle$, for all $l$. Hence, One can easily observe that $\beta$ is a group homomorphism and $Im(\beta)\le \langle b^{p}\rangle$.
	 	\item[$(ii)$] Since $\beta\in Q$, $k^{\beta(k^{\prime})} = k$. Therefore, using the Lemma \ref{s4l2} $(vi)$, $a^{l} = (a^{l})^{\beta(a)} = (a^{l})^{b^{j}} = a^{\frac{jl(l-1)}{2}((pr+1)^{\lambda p}-1)+l(pr+1)^{j}}$. Now, using the part $(i)$ and $(G3)$, we get, $l(pr+1)^{j}\equiv l \Mod{m}$, for all $l$.
	 	\item[$(iii)$] First, note that $a^{l}\cdot b^{p} = b^{p}$ and using the part $(ii)$, $(a^{l})^{b^{p}} = a^{l}$, for all $l$. Hence, the result follows using the part $(i)$.
	 	
	 	\item[$(iv)$] Using the Lemma \ref{s4l3} $(i)$, we have, $\gamma(b^{u}) = a^{\mu \sum_{\nu=0}^{u-1}{(pr+1)^{i\nu}}}$, for all $u$. Then, using the part $(ii)$, for all $l$, we get,
	 	\[\gamma\beta(a^{l}) = \gamma(b^{lj}) = a^{\mu \sum_{\nu=0}^{lj-1}{(pr+1)^{i\nu}}} = a^{{\mu}\left(\frac{(pr+1)^{ijl}-1}{(pr+1)^{i}-1}\right)} = 1.\]
	 	Thus, $\gamma\beta = 0$.
	 	\item[$(v)$] Follows directly using the part $(iv)$.
	 	\item[$(vi)$] Using $\beta(k^{h}) = \beta(k)$ and the part $(i)$, $\beta\gamma(hh^{\prime}) = \beta(\gamma(h)^{\alpha(h^{\prime})}\gamma(h^{\prime})) = \beta(\gamma(h)^{\alpha(h^{\prime})}) \beta(\gamma(h^{\prime})) = \beta(\gamma(h))\beta(\gamma(h^{\prime}))$.   Hence, $\beta\gamma \in Hom(K,K)$. 
	 	\item[$(vii)$] Using the Lemma \ref{s4l3} $(i)$, we have, $\gamma(b^{u}) = a^{\mu \sum_{\nu=0}^{u-1}{(pr+1)^{i\nu}}}$, for all $u$. Also, using the part $(i)$, we have, $\beta\gamma(b^{u}) = b^{uj\mu}$, for all $u$. Therefore,
	 	\[(\alpha+\beta\gamma)(b^{u}) = b^{u(i+j\mu+p\mu\lambda\frac{u-1}{2})}.\]
	 \n  Now, one can easily observe that $\alpha +\beta\gamma$ is a bijection. Hence, using the part $(vi)$, $\alpha+\beta\gamma \in Aut(H)$.
	 	
	 	\vspace{.2cm}
	 	
	 	\n Now, using the part $(i)$, $(C5)$ and $(C6)$, $k\cdot (\alpha+\beta\gamma)(h) = k\cdot \alpha(h)\beta\gamma(h) = (k\cdot \alpha(h))(k^{\alpha(h)}\cdot \beta(\gamma(h)) = \alpha(k\cdot h)\beta(\gamma(h) = \alpha(k\cdot h)\beta\gamma(k\cdot h) = (\alpha+\beta\gamma)(k\cdot h)$ and $k^{(\alpha+\beta\gamma)(h)} = k^{\alpha(h)\beta\gamma(h)} = (k^{\alpha(h)})^{\beta\gamma(h)} = k^{\alpha(h)} = k^{h}$. Hence, $\alpha+\beta\gamma \in P$.
	 
	 \end{itemize}
\end{proof}
\n Note that, using the Lemma \ref{s4l4} $(iii)$, multiplication in the group $\mathcal{A}$ reduces to the usual multiplication of matrices.

\begin{theorem}\label{s4t1}
	Let $A,B,C,$ and $D$ be defined as above. Then $Aut(G) = ABCD$.
\end{theorem}
\begin{proof}
		Using the Lemma \ref{s4l4} $(vii)$, $\alpha+\beta\gamma \in P$. In particular, $1-\beta\gamma \in P$. Therefore, by the Theorem \ref{s2t1}, we have, $Aut(G) = ABCD$. 
		\end{proof}
\begin{theorem}\label{s4t2}
	Let $G$ be as above. Then
\begin{equation*}
|Aut(G)| = \left\{\begin{array}{ll}
p^{2}m\frac{\phi(m)}{p-1}, & \text{if}\; (pr+1)^{p}\equiv 1 \Mod{m}\\
pm\frac{\phi(m)}{p-1}, & \text{if}\; (pr+1)^{p}\not\equiv 1 \Mod{m}
\end{array}\right..
\end{equation*}
\end{theorem}
\begin{proof}
	Let $\beta \in Q$. Then using the Lemma \ref{s4l4} $(i)$, we have that $\beta(a^{l}) = b^{jl}$, where $j\equiv 0 \Mod{p}$. Thus, $B \simeq \mathbb{Z}_{p}$. Now, let $(\alpha,\gamma,\delta)\in X$ be such that $\alpha(b) = b^{i}$, $\gamma(b) = a^{\mu}$ and $\delta(a) = a^{s}$, where $i\in \mathbb{Z}_{p^{2}}$, $\gcd(i,p^{2}) = 1$, $0\le \mu \le m-1$, and $s\in U(m)$. Then using $\alpha(hh^{\prime}) = \alpha(h)(\gamma(h)\cdot \alpha(h^{\prime}))$, $\gamma(hh^{\prime}) = \gamma(h)^{\alpha(h^{\prime})}\gamma(h^{\prime})$ and the Lemma \ref{s4l3} $(i)$, we have 
	\begin{equation}\label{s4e2}
	\alpha(b^{u}) = b^{i(u+ \frac{u(u-1)}{2}p\mu\lambda)} \;\text{and}\; \gamma(b^{u}) = a^{\mu \sum_{\nu=0}^{u-1}{(pr+1)^{i\nu}}}.
	\end{equation}

\n	Now, using $(\alpha,\gamma,\delta)\in X$, $\delta(k)\cdot \alpha(h) = \alpha(k\cdot h)$, $b^{it} = \alpha(b^{t}) = \alpha(a\cdot b) = \delta(a)\cdot \alpha(b) = a^{s}\cdot b^{i} = b^{it^{s}}$. Thus, $it^{s} \equiv it \Mod{p^{2}}$ which implies that $(1+p\lambda)^{s-1} \equiv 1 \Mod{p^{2}}$. Therefore, $s\equiv 1 \Mod{p}$. Using $(\alpha,\gamma,\delta)\in X$, $\delta(k)^{\alpha(h)}\gamma(h) = \gamma(k\cdot h)\delta(k^{h})$, $(G3)$ and the fact that $s\equiv 1 \Mod{p}$, we get, $a^{\mu \sum_{\nu=0}^{t-1}{(pr+1)^{i\nu}}+s(pr+1)}  = \gamma(b^{t})\delta(a^{pr+1}) = \gamma(a\cdot b)\delta(a^{b}) = \delta(a)^{\alpha(b)}\gamma(b) = (a^{s})^{b^{i}}a^{\mu}$ $= a^{\frac{is(s-1)}{2}((pr+1)^{\lambda p}-1)+s(pr+1)^{i}}a^{\mu} = a^{s(pr+1)^{i} + \mu}$. Thus $\mu \sum_{\nu=0}^{t-1}{(pr+1)^{i\nu}}+s(pr+1)\equiv s(pr+1)^{i}+\mu \Mod{m}$. Therefore,
	\begin{align*}
	 \mu +s(pr+1)^{i} &\equiv \mu\left(\frac{(pr+1)^{it}-1}{(pr+1)^{i}-1}\right) +s(pr+1) \Mod{m}\\
	& \equiv \mu\left(\frac{(pr+1)^{i(1+p\lambda)}-1}{(pr+1)^{i}-1}\right) +s(pr+1) \Mod{m}\\
	& \equiv \mu\left(\frac{(pr+1)^{i}(pr+1)^{ip\lambda}-1}{(pr+1)^{i}-1}\right) +s(pr+1)\Mod{m}.
	\end{align*}
\n We consider two cases, namely $(pr+1)^{p}\equiv 1 \Mod{m}$ and  $(pr+1)^{p}\not\equiv 1 \Mod{m}$.

\vspace{.2cm}
	
	\n \textit{Case($i$).} If $(pr+1)^{p}\equiv 1 \Mod{m}$, then $\mu + s(pr+1)^{i}\equiv \mu + s(pr+1) \Mod{m}$ which implies that $i\equiv 1 \Mod{p}$. Thus in this case, the choices for the maps $\alpha$, $\gamma$ and $\delta$ are, we get, $\alpha_{i}(b) = b^{i}$, $\gamma_{\mu}(b) = a^{\mu}$, and $\delta_{s}(a) = a^{s}$, where $i\in U(p^{2})$, $i \equiv 1 \Mod{p}$, $0\le \mu \le m-1$, $s\in U(m)$, and $s\equiv 1 \Mod{p}$. 
	
	\vspace{.2cm}
	
	\n \textit{Case($ii$).} If $(pr+1)^{p}\not\equiv 1 \Mod{m}$, then using the Lemma \ref{s4l3}, $\mu \equiv 0 \Mod{p}$. Therefore, $\mu + s(pr+1)^{i}\equiv \mu + s(pr+1) \Mod{m}$ which implies that $i\equiv 1 \Mod{p}$. Thus in this case, the choices for the maps $\alpha$, $\gamma$ and $\delta$ are, we get, $\alpha_{i}(b) = b^{i}$, $\gamma_{\mu}(b) = a^{\mu}$, and $\delta_{s}(a) = a^{s}$, where $i\in U(p^{2})$, $i \equiv 1 \Mod{p}$, $0\le \mu \le m-1$, $\mu \equiv 0 \Mod{p}$, $s\in U(m)$ and $s\equiv 1 \Mod{p}$. 
	
	\vspace{.2cm}
	
	\n From both the cases $(i)$ and $(ii)$, we observe that for all $\mu$, $i\equiv 1 \Mod{p}$ and $s\equiv 1 \Mod{p}$. Using these conditions, first, we find the structure of $Aut(G)$.
	
	\vspace{.2cm}
	
	\n Since, $A\times D$ normalizes $C$, $M$ normalizes C. So, clearly, $C\triangleleft E$ and $M\cap C = \{1\}$. Now, if $\begin{pmatrix}
	\alpha & 0\\
	\gamma & \delta
	\end{pmatrix} \in E$, then
	\[\begin{pmatrix}
		\alpha & 0\\
	\gamma & \delta
	\end{pmatrix} = \begin{pmatrix}
	\alpha & 0\\
	0 & \delta
	\end{pmatrix}\begin{pmatrix}
	1 & 0\\
	\delta^{-1}\gamma & 1
	\end{pmatrix} \in MC\]
	\n Thus $E = C\rtimes M$. Now, using the Lemma \ref{s4l4} $(iii)$ and $(iv)$, we get,
	\begin{equation}\label{s4e3}
	\begin{pmatrix}
	1 & \beta\\
	0 & 1
	\end{pmatrix}\begin{pmatrix}
	\alpha & 0\\
	\gamma & \delta
	\end{pmatrix}\begin{pmatrix}
	1 & \beta\\
	0 & 1
	\end{pmatrix}^{-1} = \begin{pmatrix}
	\alpha+ \beta\gamma & (\alpha+ \beta\gamma)(-\beta) + \beta\delta\\
	\gamma & \delta
	\end{pmatrix}.
	\end{equation}
	\n Using the Lemma \ref{s4l4} $(i)$ and $(ii)$, we have
	
\vspace{.2cm}
	
	\n $((\alpha+ \beta\gamma)(-\beta) + \beta\delta)(a) =(\alpha+ \beta\gamma)(-\beta)(a)(\beta\delta)(a)=(\alpha+\beta\gamma)(b^{-j})\beta(a^{s})= \alpha(b^{-j})\beta(\gamma(b^{-j}))b^{sj}= b^{-ij}\beta(a^{\mu\sum_{\nu=0}^{-j-1}{(pr+1)^{i\nu}}})b^{sj}= b^{j(s-i)}\beta\left(a^{\mu\left(\frac{(pr+1)^{-ij}-1}{(pr+1)^{i}-1}\right)}\right)= \beta(1) = 1$.

	Thus, $(\alpha+ \beta\gamma)(-\beta) + \beta\delta = 0$. Also, using the Lemma \ref{s4l4} $(vii)$, one can easily observe that $(\alpha+\beta\gamma, \gamma, \delta)\in X$. Therefore, by the Equation (\ref{s4e2}), \[\begin{pmatrix}
	1 & \beta\\
	0 & 1
	\end{pmatrix}\begin{pmatrix}
	\alpha & 0\\
	\gamma & \delta
	\end{pmatrix}\begin{pmatrix}
	1 & \beta\\
	0 & 1
	\end{pmatrix}^{-1} = \begin{pmatrix}
	\alpha+ \beta\gamma & 0\\
	\gamma & \delta
	\end{pmatrix}\in E.\]
	\n Thus $E \triangleleft \mathcal{A}$. Clearly, $E\cap B = \{1\}$.  Now, if $\begin{pmatrix}
	\alpha & \beta \\
	\gamma & \delta
	\end{pmatrix}\in \mathcal{A}$, then using $\gamma\alpha\beta = 0$, we get
	\[\begin{pmatrix}
	\alpha & \beta \\
	\gamma & \delta
	\end{pmatrix} = \begin{pmatrix}
	\alpha & 0 \\
	\gamma & \delta
	\end{pmatrix}\begin{pmatrix}
	1 & \alpha^{-1}\beta \\
	0 & 1
	\end{pmatrix}\in EB\]	
\n	Hence, $\mathcal{A} = E\rtimes B$ and so, $Aut(G)\simeq E\rtimes B\simeq (C\rtimes (A\times D))\rtimes B$.

\vspace{.2cm}

\n Thus, $|X|= p\times m\times\frac{\phi(m)}{p-1} = pm\frac{\phi(m)}{p-1}$ and $|Aut(G)| = |X||B| =  pm\frac{\phi(m)}{p-1}\times p = p^{2}m\frac{\phi(m)}{p-1}$ in the \textit{case($i$)}, and $|X|= p\times\frac{m}{p}\times\frac{\phi(m)}{p-1} = m\frac{\phi(m)}{p-1}$, $|Aut(G)| = |X||B| =  m\frac{\phi(m)}{p-1}\times p = pm\frac{\phi(m)}{p-1}$ in the \textit{case($ii$)}.
\end{proof}
\n Note that, in the Theorem \ref{s4t2}, we have $B \simeq \mathbb{Z}_{p}$, $\langle \alpha \rangle \simeq \mathbb{Z}_{p}$, $\langle \gamma \rangle \simeq \mathbb{Z}_{m}$ or $\langle \gamma \rangle \simeq \mathbb{Z}_{\frac{m}{p}}$.


\begin{thebibliography}{00}
		
		\bibitem{za} Z. Arad and E. Fisman, On Finite Factorizable Groups, \textit{J. Algebra}, 86 pp. 522-548, 1984.
		
		\bibitem{crn2008} M. J. Curran, Automorphisms of semidirect products, \textit{Math. Proc. R. Ir. Acad.}, 108A, pp. 205 - 210, 2008. 
		
		\bibitem{af} A. Firat and C. Sinan, Knit Products of Some Groups and Their Applications, \textit{Rend. Semin. Mat. Univ. Padova}, 121 pp. 1-11, 2009.
		
		\bibitem{ph} P. Hall, A Characteristic Property of Soluble Groups, \textit{J. London Math. Soc.} 12, no. 3, 198–200, 1937.
		
		\bibitem{nch} N. C. Hsu, The group of automorphisms of the holomorph of a group, \textit{Pacific J. Math.}, 11 (1961), pp. 999-1012.
		
		\bibitem{rl} R. Lal, Trnasversls in Groups, \textit{J. Algebra}, 181, pp. 70 - 81, 1996. 
				
		\bibitem{sz2}	J. Sz\'{e}p, Uber die als Produkt zweier Untergruppen darstellbaren endlichen Gruppen, \textit{Comment. Math. Helv.} 22 (1949), 31–33 (1948).	
		
		
		\bibitem{sz3} J. Sz\'{e}p,	On the structure of groups which can be represented as the product of two subgroups, \textit{Acta Sci. Math. (Szeged)} 12 (1950), 57–61.
		
		\bibitem{sz1} J. Sz\'{e}p, Zur Theorie der faktorisierbaren Gruppen, \textit{Acta Sci. Math. (Szeged)}, 16 (1955), 54–57.
		
		\bibitem{sz4} J. Sz\'{e}p and G. Zappa, Sui gruppi trifattorizzabili, \textit{Atti Accad. Naz. Lincei Rend. Cl. Sci. Fis. Mat. Nat.} (8) 45 (1968), 113–116.
		
		\bibitem{y4m} K. R. Yacoub,	On general products of two finite cyclic groups one being of order 4. (Arabic summary), \textit{Proc. Math. Phys. Soc. Egypt}, 21, pp. 119 - 126, 1957.
		
		\bibitem{ypm} K. R. Yacoub,	On general products of two finite cyclic groups one of which being of order $p^{2}$, \textit{Publ. Math. Debrecen}, 6, pp. 26 - 39, 1959.
		
		\bibitem{gz} G. Zappa, Sulla costruzione dei gruppi prodotto di due dati sottogruppi permutabili tra loro, \textit{Atti Secondo Congresso Un. Mat. Ital.}, Bologna, 1940, Edizioni Cremonense (Rome 1942), pp. 119–125.
	\end{thebibliography}
\end{document}